\documentclass[10pt]{article}

\usepackage{amsfonts} 
\usepackage{amsmath,hhline,latexsym,mathrsfs,amsthm} 
\usepackage{amssymb}
\usepackage{relsize}
\DeclareMathAlphabet\mathbfcal{OMS}{cmsy}{b}{n}
\usepackage{mathtools}

\usepackage{hyperref,enumitem}

\usepackage{color}
\usepackage{graphicx}
\usepackage{comment}

\usepackage[latin1]{inputenc} 
\usepackage[english]{babel} 

\usepackage[colored]{shadethm} 

\definecolor{shadethmcolor}{rgb}{1,0.871,0.790}

\definecolor{shaderulecolor}{rgb}{0.651,0.074,0.090}
\newshadetheorem{theorem}{Theorem} 
\newshadetheorem{prop}{Proposition} 
\newshadetheorem{lemma}{Lemma} 
\newshadetheorem{remark}{Remark} 
\newshadetheorem{defi}{Definition}

\newcommand{\dis}{\displaystyle}
\newcommand{\R}{\mathbb{R}}
\newcommand{\N}{\mathbb{N}}
\newcommand{\A}{\mathcal{A}}
\newcommand{\ac}{\mathcal{A}_0}

\begin{document}

\title{Constructive exact control of semilinear 1D heat equations}

\author{J\'er\^ome Lemoine\thanks{Laboratoire de math\'ematiques Blaise Pascal, Universit\'e Clermont Auvergne, UMR CNRS 6620, Campus des C\'ezeaux, 3, place Vasarely, 63178 Aubi\`ere, France. e-mail: jerome.lemoine@uca.fr.}
\quad\qquad \ Arnaud M\"unch\thanks{Laboratoire de math\'ematiques Blaise Pascal, Universit\'e Clermont Auvergne, UMR CNRS 6620, Campus des C\'ezeaux, 3, place Vasarely, 63178 Aubi\`ere, France. e-mail: arnaud.munch@uca.fr (Corresponding author).}
 }

\maketitle

\begin{abstract}


The exact distributed controllability of the semilinear heat equation $\partial_{t}y-\Delta y + g(y)=f \,1_{\omega}$ posed over multi-dimensional and bounded domains, assuming that $g\in C^1(\mathbb{R})$ satisfies the growth condition $\limsup_{r\to \infty} g(r)/(\vert r\vert \ln^{3/2}\vert r\vert)=0$ has been obtained by Fern\'andez-Cara and Zuazua in 2000. The proof based on a non constructive fixed point arguments makes use of precise estimates of the observability constant for a linearized heat equation. 
In the one dimensional setting, assuming that $g^\prime$ does not grow faster than $\beta \ln^{3/2}\vert r\vert$ at infinity for $\beta>0$ small enough and that $g^\prime$ is uniformly H\"older continuous on $\mathbb{R}$ with exponent $p\in [0,1]$, we design a constructive proof yielding an explicit sequence converging to a controlled solution  for the semilinear equation, at least with order $1+p$ after a finite number of iterations.
\end{abstract}
 
\textbf{AMS Classifications:} 35K58, 93B05.

\textbf{Keywords:} Semilinear heat equation,  Null controllability, Least-squares approach.

\section{Introduction}

Let $\Omega=(0,1)$, $\omega\subset\subset \Omega$ be any non-empty open set and let $T>0$. We set $Q_T=\Omega\times (0,T)$, $q_T=\omega\times (0,T)$ and $\Sigma_T=\partial\Omega\times (0,T)$.  We are concerned with the null controllability problem for the following semilinear heat equation
\begin{equation}
\label{heat-NL}
\left\{
\begin{aligned}
& \partial_ty - \partial_{xx} y +  g(y)= f 1_{\omega} \quad  \textrm{in}\quad Q_T,\\
& y=0 \,\,\, \textrm{on}\,\,\, \Sigma_T, \quad y(\cdot,0)=u_0 \,\,\, \textrm{in}\,\,\, \Omega,
\end{aligned}
\right.
\end{equation}
where $u_0\in H^1_0(\Omega)$ is the initial state of $y$ and $f\in L^2(q_T)$ is a {\it control} function. We assume moreover that the nonlinear function $g:\mathbb{R} \mapsto \mathbb{R}$ is, at~least, locally Lipschitz-continuous and, following \cite{EFC-EZ}, that $g$ satisfies
   \begin{equation}
   \vert  g^{\prime}(r) \vert \leq C (1+\vert r\vert^5) \quad \forall r\in \mathbb{R}. \label{cond_f_5}
   \end{equation}
   Under this condition, (\ref{heat-NL}) possesses exactly one local in time solution.
   Moreover, we recall (see ~\cite{CazenaveHaraux}) that under the growth condition
   \begin{equation}
\vert g(r)\vert \leq C(1+\vert r\vert \ln(1+\vert r\vert)) \quad \forall r\in \mathbb{R}, \label{growth-f}
   \end{equation}
   the solutions to (\ref{heat-NL}) are globally defined in $[0,T]$ and one has
   \begin{equation}\label{globalT}
   y \in C^0([0,T]; H_0^1(\Omega))\cap L^2(0,T; H^2(\Omega)).
   \end{equation}
  Without a growth condition of the kind (\ref{growth-f}), the solutions to~(\ref{heat-NL}) can blow up before $t=T$;
   in general, the blow-up time depends on~$g$ and the size of $\Vert u_0\Vert_{L^2(\Omega)}$.

  The system (\ref{heat-NL}) is said to be {\it controllable} at time $T$ if, for any $u_0\in L^2(\Omega)$ and any globally defined bounded trajectory $y^{\star} \in C^0([0,T]; L^2(\Omega))$
   (corresponding to data $u_0^\star \in L^2(\Omega)$ and $f^{\star}\in L^2(q_T)$), there exist controls $f \in L^2(q_T)$ and associated states~$y$ that are again globally defined in $[0,T]$ and satisfy~\eqref{globalT} and
   \begin{equation}\label{stateT}
   y(x,T) = y^{\star}(x,T), \quad x \in \Omega.
   \end{equation}
 The uniform controllability strongly depends on the nonlinearity $g$. Assuming a growth condition on the nonlinearity $g$ at infinity, this problem has been solved by Fern\'andez-Cara and Zuazua in \cite{EFC-EZ} (which also covers the multi-dimensional case for which $\Omega\subset \R^d$ is a bounded connected open set with Lipschitz boundary).
\begin{theorem}\label{nullcontrollheatplus}\cite{EFC-EZ} 
Let $T>0$ be given. Assume that $(\ref{heat-NL})$ admits at least one solution $y^{\star}$, globally defined in $[0,T]$ and bounded in $Q_T$.
   Assume that $g:\mathbb{R} \mapsto \mathbb{R}$ is $C^1$ and satisfies $(\ref{cond_f_5})$ and
\begin{enumerate}[label=$\bf (H_1)$,leftmargin=1.5cm]
\item\label{asymptotic_g}\ $\limsup_{\vert r\vert \to \infty} \frac{\vert g(r)\vert }{\vert r\vert \ln^{3/2}\vert r\vert}=0$
\end{enumerate}
then \eqref{heat-NL} is controllable in time $T$.
\end{theorem}
Therefore, if $|g(r)|$ does not grow at infinity faster than $\vert r \vert \ln^{p}(1+\vert r\vert)$ for any $p<3/2$, then (\ref{heat-NL}) is controllable. 
 We also mention \cite{VB1} which gives the same result assuming additional sign condition on $g$, namely $g(r)r\geq -C (1+r^2)$ for all $r\in \mathbb{R}$ and some $C>0$. On the contrary, if $g$ is too ``super-linear" at infinity, precisely, if $p>2$, then for some initial data, the control cannot compensate the blow-up phenomenon occurring in $\Omega \backslash \overline{\omega}$ (see \cite[Theorem 1.1]{EFC-EZ}). The problem remains open when $g$ behaves at infinity like $\vert r\vert \ln^p(1+\vert r\vert)$ with $3/2\leq p\leq 2$. We mention however the recent work of Le Balc'h \cite{KLB} where uniform controllability results are obtained for $p\leq 2$ assuming additional sign conditions on $g$, notably that $g(r)>0$ for $r>0$ or $g(r)<0$ for $r<0$, a condition not satisfied for $g(r)=-r\, \ln^p(1+\vert r\vert)$. 
Eventually, we also mention \cite{coron-trelat} where a positive boundary controllability result is proved for a specific class of initial and final data and $T$ large enough.

   In the sequel, for simplicity, we shall assume that $g(0)=0$ and that $f^\star\equiv 0, u_0^\star\equiv 0$ so that $y^\star$ is the null trajectory. The proof given in \cite{EFC-EZ} is based on a fixed point method, initially introduced in \cite{Zuazua_WaveNL} for a one dimensional wave equation.  Precisely, it is shown that the operator $\Lambda:L^\infty(Q_T)\to L^\infty(Q_T)$, where $y:=\Lambda(z)$ is a null controlled solution of the linear boundary value problem 
\begin{equation}
\label{NL_z}
\left\{
\begin{aligned}
& \partial_t y - \partial_{xx} y +  y \,\tilde{g}(z)= f 1_{\omega} \quad  \textrm{in}\quad Q_T\\
& y=0 \,\, \textrm{on}\,\, \Sigma_T, \quad y(\cdot,0)=u_0 \quad \textrm{in}\quad \Omega
\end{aligned}
\right., 
\qquad
\tilde{g}(r):=
\left\{ 
\begin{aligned}
& g(r)/r & r\neq 0\\
& g^{\prime}(0) & r=0
\end{aligned}
\right.
\end{equation}
maps a closed ball $B(0,M) \subset L^\infty(Q_T)$ into itself, for some $M>0$. The Kakutani's theorem then provides the existence of at least one fixed point for the operator $\Lambda$, which is also a controlled solution for (\ref{heat-NL}). The control of minimal $L^\infty(q_T)$ is considered in \cite{EFC-EZ}. This allows, including in the multi-dimensional case to obtain controlled solutions in $L^\infty(Q_T)$.

The main goal of this work is to determine an approximation of the controllability problem associated with~(\ref{heat-NL}), that is to construct an explicit sequence $(f_k)_{k\in \mathbb{N}}$ converging strongly toward a null control for~(\ref{heat-NL}). A natural strategy is to take advantage of the method used in \cite{KLB, EFC-EZ} and consider, for any element $y_0\in L^\infty(Q_T)$, the Picard iterations defined by $y_{k+1}=\Lambda(y_k)$, $k\geq 0$ associated with the operator $\Lambda$. The resulting sequence of controls $(f_k)_{k\in \mathbb{N}}$ is then so that $f_{k+1}\in L^2(q_T)$ is a null control for $y_{k+1}$ solution of 
\begin{equation}
\label{NL_z_k}
\left\{
\begin{aligned}
& \partial_t y_{k+1} - \partial_{xx} y_{k+1} +  y_{k+1} \,\widetilde{g}(y_k)= f_{k+1} 1_{\omega} \quad  \textrm{in}\quad Q_T,\\
& y_{k+1}=0 \,\,\, \textrm{on}\,\,\, \Sigma_T, \quad y_{k+1}(\cdot,0)=u_0 \,\,\, \textrm{in}\,\,\, \Omega.
\end{aligned}
\right. 
\end{equation}
Numerical experiments reported in \cite{EFC-AM} exhibit the non convergence of the sequences $(y_k)_{k\in\mathbb{N}}$ and  $(f_k)_{k\in\mathbb{N}}$ for some initial conditions large enough. This phenomenon is related to the fact that the operator $\Lambda$ is in general not contracting, even if $\widetilde{g}$ is globally Lipschitz. We also refer to \cite{BoyerCanum2012,Boyer-LeRousseau2014} where this strategy is implemented. 
A least-squares type approach, based on the minimization over $L^2(Q_T)$ of the functional $\mathcal{R}:L^2(Q_T)\to \mathbb{R}^+$ defined by $\mathcal{R}(z):=\Vert z-\Lambda(z)\Vert^2_{L^2(Q_T)}$
has been introduced and analyzed in \cite{EFC-AM}. Assuming that $\widetilde{g}\in C^1(\mathbb{R})$ and $g^{\prime}\in L^\infty(\mathbb{R})$, it is proved that $\mathcal{R}\in C^1(L^2(Q_T);\mathbb{R}^+)$ and that, for some constant $C>0$
$$
\Vert \mathcal{R}^\prime(z)\Vert_{L^2(Q_T)}\geq (1-C\Vert g^\prime\Vert_\infty \Vert u_0\Vert_\infty)\sqrt{2\mathcal{R}(z)} \quad \forall z\in L^2(Q_T)
$$
implying that if $\Vert g^\prime\Vert_\infty \Vert u_0\Vert_\infty$ is small enough, then any critical point for $\mathcal{R}$ is a fixed point for $\Lambda$ (see \cite[Proposition 3.2]{EFC-AM}). Under this assumption on the data, numerical experiments reported in \cite{EFC-AM} display the convergence of gradient based minimizing sequences for $\mathcal{R}$ and a better behavior than the Picard iterates. The analysis of convergence is however not performed. As is usual for nonlinear problems and also considered in \cite{EFC-AM}, we may employ a Newton type method to find a zero of the mapping $\widetilde{F}: Y \mapsto W$ defined by
   \begin{equation}\label{def-F}
\widetilde{F}(y,f) = (\partial_t y - \partial_{xx} y  + g(y) - f 1_{\omega}, y(\cdot\,,0) - u_0) \quad \forall (y,f) \in Y
   \end{equation}
 where the Hilbert space $Y$ and $W$ are defined as follows
   $$
\begin{aligned}
Y := & \biggl\{\, (y,f) : \rho y\in L^2(Q_T), \ \rho_0(\partial_t y - \partial_{xx} y)  \in L^2(Q_T), y = 0\  \hbox{on}  \ \Sigma_T, \ \rho_0 f \in L^2(q_T) \,\biggr\}
\end{aligned}
   $$
and $W := L^2(\rho_0;Q_T) \times  L^2(\Omega)$ for some appropriates weights $\rho_i$ (defined in the next section). Here $L^2(\rho_0;Q_T)$ stands for $\{z: \rho_0 z\in L^2(Q_T) \}$. It is shown in \cite{EFC-AM} that, if $g\in C^1(\mathbb{R})$ and $g^{\prime}\in L^\infty(\mathbb{R})$, then $\widetilde{F}\in C^1(Y;W)$ allowing to derive the Newton iterative sequence: given $(y_0,f_0)$ in $Y$, define the sequence $(y_k,f_k)_{k\in \mathbb{N}}$ in $Y^{\mathbb{N}}$ iteratively as follows $(y_{k+1},f_{k+1})=(y_k,f_k)-(Y_k,F_k)$ where $F_k$ is a control for $Y_k$ solution of  
\begin{equation}
\label{Newton-nn}
\left\{
\begin{aligned}
&\partial_t Y_k-\partial_{xx} Y_{k} +  g^{\prime}(y_k)\,Y_{k} = F_{k}\, 1_{\omega} + \partial_t y_{k} - \partial_{xx} y_k  + g(y_k) - f_k 1_{\omega}
                                                                                    & \quad  \textrm{in}\quad Q_T,\\
& Y_k=0                                               \quad  \textrm{on}\quad \Sigma_T, \quad Y_k(\cdot,0)=u_0-y_k(\cdot,0) \quad  \textrm{in}\quad \Omega.
\end{aligned}
\right.
   \end{equation}
Numerical experiments in \cite{EFC-AM} exhibits the lack of convergence of the Newton method for large enough initial condition, for which the solution $y$ is not close enough to the zero trajectory. 
 
   The controllability of nonlinear partial differential equations has attracted a large number of works in the last decades (see the monography \cite{coron-book} and references therein).
However, as far as we know, few are concerned with the approximation of exact controls for nonlinear partial differential equations, and the construction of convergent control approximations for nonlinear equations remains a challenge. 

In this article, given any initial data $u_0\in H_0^1(0,1)$, we design an algorithm providing a sequence $(f_k)_{k\in \N}$ converging to a controlled solution for \eqref{heat-NL}, under assumptions on $g$ that are slightly stronger than \ref{asymptotic_g}. Moreover, after a finite number of iterations, the convergence is super-linear. This is done (following and improving \cite{AM-PP-2014} devoted to a linear case) by introducing a quadratic functional measuring how much a pair $(y,f) \in Y$ is close to a controlled solution for \eqref{heat-NL} and then by determining a particular minimizing sequence enjoying the announced property. A natural example of an error (or least-squares) functional is given by $\widetilde{E}(y,f):=\frac{1}{2}\Vert \widetilde{F}(y,f)\Vert^2_W$ to be minimized over $Y$. The controllability for~(\ref{heat-NL}) is reflected by the fact that the global minimum of the nonnegative functional $\widetilde{E}$ is zero, over all pairs $(y,f)\in Y$ solutions of \eqref{heat-NL}. 

  The paper is organized as follows. In Section \ref{sec_control_linearized}, we first derive a controllability result for a linearized wave equation with potential in $L^\infty(Q_T)$ and source term in $L^2(Q_T)$. Then, in Section \ref{sec_LS}, we define the least-squares functional $E$ and the corresponding non convex optimization problem \eqref{extremal_problem} over the Hilbert space $\mathcal{A}$. We show that $E$ is Gateaux-differentiable over $\mathcal{A}$ and that any critical point $(y,f)$ for $E$ for which $g^\prime(y)$ belongs to $L^\infty(Q_T)$ is also a zero of $E$ (see Proposition \ref{proposition4}). This is done by introducing a pair $(Y^1,F^1)$ for $E(y,f)$ for which $E^\prime(y,f)\cdot (Y^1,F^1)$ is proportional to $E(y,f)$. Then, in Section \ref{sec_convergence},assuming that the nonlinear function $g$ is such that $g^\prime$ is uniformly Holder continuous with exponent $p$, for some $p\in [0,1]$, we determine a minimizing sequence based on $(Y^1,F^1)$ which converges strongly to a controlled pair for the semilinear heat equation (\ref{heat-NL}). Moreover, we prove that after a finite number of iterates, the convergence enjoys a rate equal to $1+p$ (see Theorem \ref{ths1}). Section \ref{sec:remarks} gathers several remarks on the approach: we notably emphasize that this least-squares approach coincides with the damped Newton method one may use to find a zero of a mapping similar to $\widetilde{F}$ mentioned above: this explains the super-linear convergence obtained. We also discuss some other linearizations of the system (\ref{heat-NL}). We conclude in Section \ref{conclusion} with some perspectives.

  As far as we know, the method introduced and analyzed in this work is the first one providing an explicit, algorithmic construction of exact controls for semilinear heat equations with non Lipschitz nonlinearity. It extends the study \cite{lemoine_gayte_munch} assuming that $g^\prime\in L^\infty(\mathbb{R})$ which to obtain directly a uniform bound of the observability constant. The weaker assumption considered here required a refined analysis similar to the one recently developed by the author in \cite{munch_lemoine_waveND, munch_trelat} for the wave equation. The parabolic case is however more much intricate (than the hyperbolic one) as it makes appear Carleman type weights depending on the controlled solution. These works devoted to controllability problems take their roots in the works \cite{lemoinemunch_time, lemoinemunch_NUMER} concerned with the approximation of solution of Navier-Stokes type problem, through least-square methods: they refine the analysis performed in \cite{lemoine-Munch-Pedregal-AMO-20, MunchMCSS2015} inspired from the seminal contribution \cite{Bristeau1979}.

\paragraph{Notations.}
Throughout, we denote by $\Vert \cdot \Vert_{\infty} $ the usual norm in $L^\infty(\R)$, by $(\cdot,\cdot)_X$ the scalar product of $X$ (if $X$ is a Hilbert space) and by $\langle \cdot, \cdot \rangle_{X,Y}$ the duality product between $X$ and $Y$.

Given any $p\in [0,1]$, we introduce for any $g\in C^1(\R)$ the following hypothesis : 
\begin{enumerate}[label=$\bf (\overline{H}_p)$,leftmargin=1.5cm]
\item\label{constraint_g_holder}\ $[g^\prime]_p := \sup_{a,b\in \R \atop a\neq b} \frac{\vert g^\prime(a)-g^\prime(b)\vert}{\vert a-b\vert^p}  < +\infty$
\end{enumerate}
meaning, for $p\in (0,1]$, that $g^\prime$ is uniformly H\"older continuous with exponent $p$.
For $p=0$, by extension, we set $[g']_0 :=2\Vert g^\prime\Vert_\infty$.
In particular, $g$ satisfies $\bf (\overline{H}_0)$ if and only if $g\in \mathcal{C}^1(\R)$ and $g^\prime\in L^\infty(\R)$, and $g$ satisfies $\bf (\overline{H}_1)$ if and only if $g^\prime$ is Lipschitz continuous (in this case, $g^\prime$ is almost everywhere differentiable and $g^{\prime\prime}\in L^\infty(\R)$), and we have $[g^\prime]_1\leq \Vert g^{\prime\prime}\Vert_\infty$.

We also denote by $C$ a positive constant depending only on $\Omega$ and $T$ that may vary from lines to lines.

\section{A controllability result for a linearized heat equation with $L^2$ right hand side}\label{sec_control_linearized}

This section is devoted to a controllability result for a linear heat equation with potential in $A\in L^\infty(Q_T)$ and right hand side $B\in L^2(\rho_0(s),Q_T)$  for a precise weight $\rho_0(s)$ parametrized by $s\in \mathbb{R^\star_+}$ defined in the sequel. More precisely we are interested by the existence of a control $v$ such that the solution $z$ of 
\begin{equation}
\label{heat_z}
\left\{
\begin{aligned}
&\partial_t z - \partial_{xx} z +  A z  = v 1_{\omega} + B \quad \textrm{in}\quad Q_T,\\
& z=0 \,\, \textrm{on}\,\, \Sigma_T, \quad z(\cdot,0)=z_0 \,\, \textrm{in}\,\, \Omega
\end{aligned}
\right.
\end{equation}
satisfies
\begin{equation}
\label{heat_z1}
z(\cdot,T)=0  \hbox{ in }\Omega.
\end{equation}

As this work concerns the null controllability of parabolic equation, 
we make use  of Carleman weights introduced in this context in \cite{fursikov-imanuvilov}). For any $s\geq 0$, we consider the weight functions $\rho(s)=\rho(x,t,s)$, $\rho_0(s)=\rho_0(x,t,s)$ and $\rho_1(s)=\rho_1(x,t,s)$  which are continuous, strictly positives and in $L^\infty(Q_{T-\delta})$ for any $\delta>0$. Precisely,  we use  the weights introduced in \cite{MB-SE-SG}: $\rho_0(s)=\xi^{-3/2} \rho(s)$, $\rho_1(s)=\xi^{-1}\rho(s)$  where $\rho(s)$ and $\xi$ are defined, for all $s\ge 1$ and $\lambda\ge 1$, as follows
  \begin{equation}\label{def_weight}
  \rho(x,t,s)=\exp\Big(s\varphi(x,t)\Big), \quad \xi(x,t)=\theta(t)\exp(\lambda \widehat{\psi}(x)) 
  \end{equation}
  where $\theta\in \mathcal{C}^2([0,T))$ is defined such that, noting $\mu=s\lambda^2e^{2\lambda}$ and $0<T_1<\min(\frac{1}{4},\frac{3T}{8})$,
   \begin{equation}\label{def_xi}
   \theta(t)=\left\{\begin{aligned}
   &1+\biggl(1-\frac{4t}T\biggr)^\mu \quad\forall t\in[0,T/4]\\
   & 1\quad \forall t\in[T/4,T-2T_1]\\
   &\theta \hbox{ is increasing on }[T-2T_1,T-T_1],   \\
   &\frac{1}{T-t}\quad\forall t\in [T-T_1,T)
   \end{aligned}\right.
\end{equation}
and $\varphi\in \mathcal{C}^1([0,T))$ is defined by
\begin{equation}
\varphi(x,t)=\theta(t)\big(\lambda  \exp(12\lambda )-\exp(\lambda \widehat{\psi}(x))\big) \label{def_varphi}
\end{equation}
with  $\widehat{\psi}=\widetilde\psi+6$, where $\widetilde\psi\in \mathcal{C}^1(\overline{\Omega})$ satisfies $\widetilde\psi\in (0,1)$ in $\Omega$, $\widetilde\psi=0$ on $\partial\Omega$ and $|\partial_x\widetilde\psi(x)|>0$ in $\overline{\Omega\backslash\omega}$. We emphasize that the weights blow up at $t\to T^-$. 

\begin{remark}\label{remark-rho0}
We shall use in the sequel that $1< \rho_0\le \rho_1\le\rho$. Indeed, since $\xi\ge 1$, $ \rho_0\le \rho_1\le\rho$.  Moreover, for all $(x,t)\in Q_T$ and $\lambda\ge 1$, we check that 
$\varphi(x,t)\ge \frac32 \xi(x,t)$
and thus, since $s\ge 1$ and $\xi(x,t)\ge 1$, we get
$$
\rho_0(x,t,s)=\xi^{-3/2}(x,t) \rho(x,t,s)\ge \xi^{-3/2}(x,t) \exp\Big(\frac{3}2s\xi(x,t)\Big)  \ge e^{3/2s} \quad \forall (x,t)\in Q_T.
$$
\end{remark}

The controllability property for the linear system \eqref{heat_z} is based on the following Carleman estimate.

\begin{lemma}\label{carleman}
There exists  $\lambda_0\ge1$  and $s_0\ge 1$ such for all $\lambda\ge \lambda_0$ and for all $s \ge \max( \|A\|^{2/3}_{L^\infty(Q_T)},s_0)$  one has the following Carleman estimate, for all $p\in P_0=\{q\in C^2(\overline{Q_T})\ : \ q=0\hbox{ on } \Sigma_T\}$:
\begin{equation}\label{Carleman-ine1} 
\begin{aligned}
\int_\Omega\rho^{-2}(0,s)|\partial_x p(0)|^2
&+s^3\lambda^4 e^{14\lambda}\int_\Omega \rho^{-2}(0,s) |p(0)|^2
+s\lambda^2\int_{Q_T} \rho_1^{-2}(s) |\partial_x p|^2 
+s^3\lambda^4\int_{Q_T} \rho_0^{-2}(s) | p|^2   \\
&\le 
C\int_{Q_T}\rho^{-2}(s)|-\partial_t p-\partial_{xx}p+ Ap|^2+ Cs^3\lambda^4\int_{q_T}\rho_0^{-2}(s)|p|^2.
\end{aligned}
\end{equation}
\end{lemma}

\begin{proof} This estimate is deduced from the one obtained in \cite[Theorem 2.5]{MB-SE-SG} devoted to the case $A\equiv 0$: there exist $\lambda_1\ge 1$ and $s_0\ge 1$ such that for all smooth functions $z$ on $\overline{Q_T}$ satisfying $z=0$ on $\Sigma_T$ and for all $s\ge s_0$ and $\lambda\ge \lambda_1$ 
$$\begin{aligned}
\int_\Omega\rho^{-2}(0,s)|\partial_x p(0)|^2
&+s^3\lambda^4 e^{14\lambda}\int_\Omega \rho^{-2}(0,s) |p(0)|^2
+s\lambda^2\int_{Q_T} \rho_1^{-2}(s) |\partial_x p|^2 
+s^3\lambda^4\int_{Q_T} \rho_0^{-2}(s) | p|^2   \\
&\le 
C\int_{Q_T}\rho^{-2}(s)|\partial_t p+\partial_{xx}p|^2+ Cs^3\lambda^4\int_{q_T}\rho_0^{-2}(s)|p|^2.
\end{aligned}
$$
Writing that
$$\begin{aligned}
\int_{Q_T}\rho^{-2}(s)|\partial_t p+\partial_{xx}p|^2
&\le 2\int_{Q_T}\rho^{-2}(s)|-\partial_t p-\partial_{xx}p+ Ap|^2+2\int_{Q_T}\rho^{-2}(s)| A p|^2\\
&\le 2\int_{Q_T}\rho^{-2}(s)|-\partial_t p-\partial_{xx}p+ Ap|^2+2\|A\|_{L^\infty(Q_T)}^2\int_{Q_T}\rho^{-2}(s)| p|^2
\end{aligned}$$
we infer, since $\rho_0\le\rho$ 
$$\begin{aligned}
&\int_\Omega\rho^{-2}(0,s)|\partial_x p(0)|^2
+s^3\lambda^4 e^{14\lambda}\int_\Omega \rho^{-2}(0,s) |p(0)|^2
+s\lambda^2\int_{Q_T} \rho_1^{-2}(s) |\partial_x p|^2 
+s^3\lambda^4\int_{Q_T} \rho_0^{-2}(s) | p|^2   \\
&\le 
C\int_{Q_T}\rho^{-2}(s)|-\partial_t p-\partial_{xx}p+ Ap|^2+C\|A\|_{L^\infty(Q_T)}^2\int_{Q_T}\rho_0^{-2}(s)| p|^2 + Cs^3\lambda^4\int_{q_T}\rho_0^{-2}(s)|p|^2.
\end{aligned}
$$

Taking $\lambda  \ge \lambda_0=\max(\lambda_1,(2C)^{1/4})$ and $s \ge \max(\|A\|^{2/3}_{L^\infty(Q_T)},s_0)$ leads to (\ref{Carleman-ine1}).
\end{proof} 

In the sequel we assume that $\lambda=\lambda_0$ and denote by $C$ any constant depending only on $\Omega$, $\omega$, $\lambda_0$  and $T$.

\begin{theorem}\label{controllability_result}
Assume $A\in L^{\infty}(Q_T)$, $s\ge\max( \|A\|^{2/3}_{L^\infty(Q_T)},s_0)$, $B\in L^2(\rho_0(s),Q_T)$ and $z_0\in L^2(\Omega)$. Then there exists a control $v\in L^2(\rho_0(s),q_T)$ such that the weak solution $z$ of  (\ref{heat_z})
satisfies  (\ref{heat_z1}).

  Moreover, the unique control $v$ which minimizes together with the corresponding solution $z$ the functional 
$J:L^2(\rho(s),Q_T)\times L^2(\rho_0(s),q_T)\to \mathbb{R}^+$ defined by $J(z,v):=\frac12\Vert \rho(s)\,  z\Vert^2_{L^2(Q_T)} + \frac12\Vert \rho_0(s)\,  v\Vert^2_{L^2(q_T)}$
satisfies the following  estimates
\begin{equation}\label{estimation1}
\Vert \rho(s)\, z\Vert_{L^2(Q_T)}+ \Vert \rho_0(s)\,v\Vert_{L^2(q_T)} \leq  Cs^{-3/2} \big(\|\rho_0(s)B\|_{L^2(Q_T)} +e^{cs }\|z_0\|_2 \big)
\end{equation}
with $c:=\Vert \varphi(\cdot,0)\Vert_\infty$ and
\begin{equation}\label{estimation2}
\Vert \rho_1(s) z\Vert_{L^\infty(0,T;L^2(\Omega))}+ \Vert \rho_1(s)\partial_x z\Vert_{L^2(Q_T)}\le 
C_1(s,A)\big(  \|\rho_0(s) B\|_{L^2(Q_T)}+e^{cs }\|z_0\|_2 \big)
  \end{equation}
  where 
$$
C_1(s,A):= Cs^{-1/2}(1 +\|A\|_{L^\infty(Q_T)}^{1/2}).
$$
Moreover, if $z_0\in H^1_0(\Omega)$ then $z\in L^\infty(Q_T)$ and
\begin{equation}\label{estimation1bis} 
\|z\|_{L^\infty(Q_T)} \le Ce^{-\frac32s}(1+\|A\|_{L^\infty(Q_T)})\big(\|\rho_0(s)B\|_{L^2(Q_T)}+e^{cs}\|z_0\|_{H^1_0(\Omega)}\big).  
\end{equation}
\end{theorem}
\par\noindent
We refer to \cite{EFC-EZ-lin} for an estimate of the null control of minimal $L^2(q_T)$-norm (corresponding to $\rho_0\equiv 1$ and $\rho=0$) in the case $B\equiv 0$, refined later on in \cite{Duyckaerts,EFC-EZ}. 
Theorem \ref{controllability_result} is based on several technical results. Remark first that the bilinear form 
$$(p,q)_P:=\int_{Q_T}\rho^{-2}(s) L^\star_A p\,L^\star_A q+s^3\lambda_0^4\int_{q_T} \rho_0^{-2}(s)p\,q$$
where $L^\star_A q:=-\partial_t q-\partial_{xx} q+A q$ for all $q\in P_0$ is a scalar product on $P_0$ (see \cite{EFC-MUNCH-SEMA}). The completion $P$  of $P_0$ for the norm $\|\cdot\|_P$ associated with this scalar product is a Hilbert space. By density arguments, (\ref{Carleman-ine1}) remains true for all $p\in P$, that is, for $\lambda=\lambda_0$,   
\begin{equation}\label{Carleman-ine} 
\int_\Omega\rho^{-2}(0,s)|\partial_x p(0)|^2
+s^3\lambda_0^4 e^{14\lambda_0}\int_\Omega \rho^{-2}(0,s) |p(0)|^2
+s\lambda_0^2\int_{Q_T} \rho_1^{-2}(s) |\partial_x p|^2 
+s^3\lambda_0^4\int_{Q_T} \rho_0^{-2}(s) | p|^2   
\le C\|p\|_P^2
\end{equation}
for all $s\ge\max( \|A\|^{2/3}_{L^\infty(Q_T)},s_0)$.

\begin{remark}
We denote by $P$ (instead of $P_A$) the completion of $P_0$ for the norm $\|\cdot\|_P$ since $P$ does not depend on $A$ (see \cite[Lemma 3.1]{EFC-AM}).
\end{remark}

\begin{lemma}\label{solution-p}
Let      $s \ge \max( \|A\|^{2/3}_{L^\infty(Q_T)},s_0) $.
There exists $p\in P$ unique solution of 
\begin{equation}\label{solution-p1}
(p,q)_P=\int_\Omega z_0q(0)+\int_{Q_T}  B q ,\quad \forall q\in P.
\end{equation}
This solution satisfies the following estimate (with $c:=\Vert \varphi(\cdot,0)\Vert_\infty$)
\begin{equation}\label{21}
\|p\|_P\le Cs^{-3/2}\big( \Vert \rho_0(s)\,B\Vert_{L^2(Q_T)} +e^{cs}\|z_0\|_2\big) .
\end{equation}	

\end{lemma}
\begin{proof}  The linear  map $L_1: P\to\R$, $q\mapsto \int_{Q_T}   Bq$ is continuous. Indeed, for all $q\in P$
$$\Big|\int_{Q_T} Bq\Big|\le \Big(\int_{Q_T}|\rho_0(s)B|^2\Big)^{1/2}
\Big(\int_{Q_T} |\rho_0^{-1}(s)q|^2\Big)^{1/2}$$
and since from the Carleman estimate (\ref{Carleman-ine}) we have
$\displaystyle \Big(\int_{Q_T} |\rho_0^{-1}(s)q|^2\Big)^{1/2}\le
 Cs^{-3/2} \|q\|_P,$
therefore
$$|L_1(q)|=\Big|\int_{Q_T} Bq\Big|\le   Cs^{-3/2} \|\rho_0(s)B\|_{L^2(Q_T)} \|q\|_P.$$
Thus $L_1$ is continuous.

From (\ref{Carleman-ine}) we deduce that the linear map $L_2:  P\to\R$, $q\mapsto \int_\Omega z_0q(0)$ is continuous. Indeed, noting $c:=\Vert \varphi(\cdot,0)\Vert_\infty$ and using $s\ge 1$, we obtain for all $q\in P$ that:
$$\begin{aligned}
|L_2(q)|
&=  s^{-3/2} e^{cs}\|z_0\|_2s^{3/2} e^{ -cs}\|q(0)\|_2\\
&\le  s^{-3/2} e^{cs }\|z_0\|_2s^{3/2}\|q(0)e^{-s\varphi(x,0)}\|_2=  s^{-3/2} e^{cs }\|z_0\|_2s^{3/2}\|\rho^{-1}(0,s) q(0)\|_2 \\
&\le C s^{-3/2}e^{cs}\|z_0\|_2\|q\|_P.
\end{aligned}
$$

Using Riesz's theorem, we conclude that there exists exactly one solution $p\in P$ of  (\ref{solution-p1}) and this solution satisfies (\ref{21}).
\end{proof} 

Let us now introduce the convex set
$$
\mathcal{C}(z_0,T):=\biggl\{(z,v) :  \rho(s) z\in L^2(Q_T),\  \rho_0(s)v\in L^2(q_T),\\
\ (z,v) \hbox{ solves  (\ref{heat_z})-(\ref{heat_z1})  in the transposition sense}\biggr\}
$$
that is  $(z,v)$ is solution of
\begin{equation}\label{23}
\int_{Q_T}zL^\star_A q=\int_{q_T} vq+\int_\Omega z_0q(0)+\int_{Q_T}  Bq,\quad \forall q\in P.
\end{equation}
Let us remark that if $(z,v)\in \mathcal{C}(z_0,T)$, then since   $v\in L^2(q_T)$ and $B\in L^2(Q_T)$, $z$ coincides with the unique weak solution of  (\ref{heat_z}) associated with $v$. We can now claim that $\mathcal{C}(z_0,T)$ is non empty. Indeed we have :

\begin{lemma}\label{solution-controle}
Let      $s \ge \max( \|A\|^{2/3}_{L^\infty(Q_T)},s_0) $,  $p\in P$ the unique solution of (\ref{solution-p1}) given in Lemma \ref{solution-p} and $(z,v)$  defined by
\begin{equation}\label{zv}
z=\rho^{-2}(s)L^\star_A p\quad \hbox{ and }\quad v=-s^{3/2}\lambda_0^2\rho_0^{-2}(s) p|_{q_T}.\end{equation}
Then $(z,v)\in \mathcal{C}(z_0,T)$ and satisfies the following estimate (with $c:=\Vert \varphi(0,\cdot)\Vert_\infty$)
\begin{equation}\label{estimation-z-v1}
\Vert \rho(s)\, z\Vert_{L^2(Q_T)}+ \Vert \rho_0(s)\,v\Vert_{L^2(q_T)} \leq C s^{-3/2}\big(\|\rho_0(s)B\|_{L^2(Q_T)}+ e^{cs}\|z_0\|_2 \big). 
\end{equation}
\end{lemma}
\begin{proof}    From the definition of $P$,  $\rho(s)  z\in L^2(Q_T)$ and $\rho _0(s)v\in L^2(q_T)$ and from the definition of $\rho(s)$ and  $\rho_0(s)$ we have $z\in L^2(Q_T)$ and $v\in L^2(q_T)$. In  view of (\ref{solution-p1}), $(z,v)$ is solution of (\ref{23})
and satisfies (\ref{estimation-z-v1})
that is, $z$ is the solution of  (\ref{heat_z})-(\ref{heat_z1}) associated with $v$ in the transposition sense. Thus $(z,v)\in \mathcal{C}(z_0,T)$. 
\end{proof}

Let us now consider the following extremal problem, introduced by Fursikov and Imanuvilov \cite{fursikov-imanuvilov}
\begin{equation}
\left\{\begin{aligned}
\label{extrema-problem}
&\hbox{ Minimize } J(z,v)=\frac12\int_{Q_T}\rho^2(s)|z|^2+\frac12\int_{q_T}\rho_0^2(s)|v|^2\\
&\hbox{ Subject to } (z,v)\in \mathcal{C}(z_0,T).
\end{aligned}\right.
\end{equation}
Then $(z,v)\mapsto J(z,v)$ is strictly convex and continuous on $L^2(\rho^2(s);Q_T)\times L^2(\rho_0^2(s); q_T)$. Therefore  (\ref{extrema-problem}) possesses at most a  solution in $\mathcal{C}(z_0,T)$. More precisely we have :

\begin{prop}\label{minimal-control}
Let      $s \ge \max( \|A\|^{2/3}_{L^\infty(Q_T)},s_0) $. 
 Then $(z,v)\in \mathcal{C}(z_0,T)$ defined in Lemma \ref{solution-controle} is the unique solution of (\ref{extrema-problem}).
\end{prop}

\begin{proof}  Let $(y,w)\in \mathcal{C}(z_0,T)$. Since $J$ is  convex and differentiable on $L^2(\rho^2(s);Q_T)\times L^2(\rho_0^2(s); q_T)$ we have :
$$\begin{aligned}
J(y,w)
&\ge J(z,v)+\int_{Q_T} \rho^2(s)z(y-z)+\int_{q_T}\rho_0^2(s)v(w-v)\\
&=J(z,v)+\int_{Q_T} L^\star p (y-z)-\int_{q_T}p(w-v)=J(z,v)
\end{aligned}$$
$y$  being  the solution of  (\ref{heat_z}) associated with $w$ in the transposition sense. 
\end{proof} 

\begin{proof}  of Theorem \ref{controllability_result}.
Proposition \ref{minimal-control} gives the existence of a control $v\in L^2(\rho_0(s),q_T)$ such that the solution $z$ of  (\ref{heat_z})
satisfies  (\ref{heat_z1}). Moreover, this control is the unique control  which minimizes together with the corresponding solution $z$ the functional 
 $J$ and satisfies (\ref{estimation1}). To finish the proof of Theorem \ref{controllability_result}, it suffices to prove that $(z,v)$ satisfies the estimate  (\ref{estimation2}) and if $z_0\in{H^1_0(\Omega)}$, then $z\in L^\infty(Q_T)$ and satisfies the estimate  (\ref{estimation1bis}).

Multiplying (\ref{heat_z}) by $\rho_1^2(s)z$ and integrating by part we obtain
\begin{equation}\label{27}
\frac12\int_\Omega (\partial_t |z|^2)\,\rho_1^2(s)+\int_\Omega \rho_1^2(s)|\partial_x z|^2+2\int_\Omega \rho_1(s)z\partial_x \rho_1(s)\cdot\partial_x z +\int_\Omega \rho_1^2(s)Azz=\int_\omega v\rho_1^2(s) z+\int_\Omega B\rho_1^2(s) z.
\end{equation}
But $\int_\Omega (\partial_t |z|^2)\rho_1^2(s)=\partial_t \int_\Omega |z|^2\rho_1^2(s)-2\int_\Omega |z|^2\rho_1(s) \partial_t \rho_1(s)$
and $\partial_t\rho_1(s)=-\frac{\partial_t\theta}{\theta} \rho_1(s)+s \frac{\partial_t\theta}{\theta}\varphi \rho_1(s)$. From the definition of $\theta$ and $\varphi$ we have:
$$
   \Big|\frac{\partial_t \theta}{\theta}(t)\Big|\le 
   \left\{\begin{aligned}
   &Cs \quad\forall t\in[0,T/4]\\
   &0 \quad \forall t\in[T/4,T-2T_1]\\
   &\theta \quad\forall t\in [T-T_1,T)
    \end{aligned}\right.$$
and $ \Big|\frac{\partial_t \theta}{\theta}(t)\Big|\le  C \quad\forall t\in [T-2T_1,T-T_1]$ since $\dfrac{\partial_t \theta}{\theta}(T-2T_1)=0$, $\dfrac{\partial_t \theta}{\theta}(T-T_1)=\frac{1}{T_1}$ and $\theta$ is $\mathcal{C}^2$.
Since $\theta\le \xi$ and $s\ge 1$, on $ [0,T)$:
$$ \Big|\frac{\partial_t \theta}{\theta}\Big|\le Cs\xi.$$
From the definition of $\varphi$, $\varphi\le C\theta\le C\xi$ and thus on $[0,T)$ :
$$   \Big|\frac{\partial_t \theta}{\theta}\varphi \Big|\le Cs\xi^2.$$
Thus, since $s\ge 1$, $\xi\ge 1$  and $\rho(s)=\xi\rho_1(s)$, on $[0,T)$ :
$$
-\int_\Omega |z|^2\rho_1(s) \partial_t \rho_1(s)
=\int_\Omega \frac{\partial_t\theta}{\theta}  |\rho_1(s)  z|^2-s\int_\Omega\frac{\partial_t\theta}{\theta}\varphi  |\rho_1(s)  z|^2
\le  C  s^2 \int_\Omega \rho^2(s)|z|^2.
$$
On the other hand
$$\partial_x\rho_1(s)=\partial_x(\xi^{-1}\rho(s))=\partial_x(\xi^{-1})\rho(s)+\xi^{-1}\partial_x \rho(s)=-\partial_x \widehat{\psi}\big( \lambda_0\xi^{-1}\rho(s)+ s\lambda_0 \rho(s)\big)=-\partial_x \widehat{\psi} \lambda_0 \rho(s)\big(\xi^{-1}+ s\big)$$
and thus, since $\xi\ge 1$  and $s\ge 1$, we write $(\xi^{-1}+s)\leq 2s$ and 
$$\begin{aligned}
\Big|\int_\Omega \rho_1(s)z\partial_x \rho_1(s)\cdot\partial_x z \Big|
&\le 2\Vert\partial_x \hat{\psi} \Vert_\infty \lambda_0 \int_\Omega s|\rho(s) z| \, |\rho_1(s) \partial_x z|\\
&\le Cs^2  \int_\Omega |\rho(s) z|^2 + \frac12  \int_\Omega |\rho_1(s)\partial_x z|^2 .
\end{aligned}
$$
We also have, since $\rho(s)=\xi\rho_1(s)$ and $\xi\ge 1$ the estimate $\Big|\int_\Omega \rho_1^2(s) Azz\Big|\le C \|A\|_\infty\|\rho(s) z\|_2^2$.
 Finally, since $\rho_1^2(s)=\xi^{-1/2} \rho_0(s)\rho(s)$ and $\xi^{-1/2}\le 1$, we infer that
$$
\Big|\int_\omega v\rho_1^2(s)z\Big|\le \Big|\int_\omega \rho_0(s) v\xi^{-1/2}\rho(s) z\Big|\le \Big(\int_\omega|\rho_0(s) v|^2\Big)^{1/2}\|\rho(s) z\|_2
$$
and $|\int_\Omega  B\rho_1^2(s) z|\le \|\rho_0(s) B\|_{2}\|\rho(s)z\|_{2}$.
Thus (\ref{27}) implies that 
$$
\partial_t \int_\Omega\rho_1^2(s) |z|^2+\int_\Omega \rho_1^2(s)|\partial_x z|^2\le
C \big(s^2  +\|A\|_{L^\infty(Q_T)} \big)\| \rho(s) z\|_2^2
 +\Big(\Big(\int_\omega|\rho_0(s) v|^2\Big)^{1/2}+ \|\rho_0(s) B\|_{2}\Big)\|\rho(s) z\|_2\\
  $$
  and therefore for all $t\in[0,T)$, since $\|\rho_1(s,0)z_0\|_2^2\le e^{cs}\|z_0\|_2^2$ (with $c:=\Vert \varphi(\cdot,0)\Vert_\infty$), we get
$$
\begin{aligned}
\Big( \int_\Omega\rho_1^2(s) |z|^2\Big)(t)
&+\int_{Q_t} \rho_1^2(s)|\partial_x z|^2\le 
C \big(s^2+\|A\|_{L^\infty(Q_T)} \big)\| \rho(s) z\|_{L^2(Q_T)}^2\\
 &+ \big(\|\rho_0(s) v\|_{L^2(q_T)}+ \|\rho_0(s) B\|_{L^2(Q_T)} \big)\|\rho(s) z\|_{L^2(Q_T)}+e^{cs}\|z_0\|_2^2.
 \end{aligned}
$$
Using (\ref{estimation-z-v1}) we obtain, since $s\ge 1$, for all $t\in[0,T)$ : 
\begin{equation} \label{constante-pourrie}
\Big( \int_\Omega\rho_1^2(s) |z|^2\Big)(t)+\int_{Q_t} \rho_1^2(s)|\partial_x z|^2
\le 
 C s^{-1} \big(1+\|A\|_{L^\infty(Q_T)} \big)\big( \|\rho_0(s) B\|_{L^2(Q_T)}^2 +e^{cs}\|z_0\|_2^2\big) \end{equation}
  which gives (\refeq{estimation2}).
  
 If $z_0\in H^1_0(\Omega)$, since $z$ is a weak solution of (\ref{heat_z}) associated with $v$, standard arguments give that $z\in L^2(0,T;H^2(\Omega))$,  $\partial_tz\in L^2(Q_T)$ and therefore  $z\in L^\infty(Q_T)$. Moreover,  multiplying  (\ref{heat_z}) by $\partial_{xx}z$  and integrating by part we obtain 
 $$\frac12\partial_t\int_\Omega|\partial_x z|^2+\int_\Omega |\partial_{xx}z|^2\le \|-Az+v1_\omega+B\|_2\|\partial_{xx} z\|_2$$
 and thus, since $\rho(s)\ge\rho_0(s)\ge e^{\frac32s}$: 
  $$ 
  \begin{aligned} \partial_t\int_\Omega|\partial_x z|^2
  &\le \int_\Omega|-Az+v1_\omega+B|^2
  \le 3\big(\int_\Omega|Az|^2+\int_\Omega| v1_\omega |^2+\int_\Omega| B|^2\big)\\
&\le 3e^{-3s}\Big(\|A\|_{L^{\infty}(Q_T)}^2\|\rho(s) z\|_2^2+\|\rho_0(s)v1_\omega\|_2^2+\|\rho_0(s)B\|_2^2\big)
\end{aligned} $$
which gives, a.e in $t\in(0,T)$
$$ 
  \int_\Omega|\partial_x z|^2(t) 
  \le 3e^{-3s}\Big(\|A\|_{L^{\infty}(Q_T)}^2\|\rho(s) z\|_{L^2(Q_T)}^2+\|\rho_0(s)v\|_{L^2(q_T)}^2+\|\rho_0(s)B\|_{L^2(Q_T)}^2\big)+\|  z_0\|_{H^1_0(\Omega)}^2.
$$
 
 Since $z\in H^1_0(\Omega)$, a.e in $(x,t)\in\Omega\times (0,T)$ (recall that $\Omega=(0,1)$), $z(x,t)=\int_0^x\partial_x z(r,t)d r\le \|\partial_x z\|_2(t)$
 and thus, using (\ref{estimation1}), since $s\ge 1$ :
$$\begin{aligned}
\|z\|_{L^\infty(Q_T)}
&\le  \|\partial_x z\|_{L^\infty(0,T;L^2(\Omega))}\\
&\le  \sqrt{3}e^{-\frac32s}\Big(\|A\|_{L^{\infty}(Q_T)}\|\rho(s) z\|_{L^2(Q_T)}+\|\rho_0(s)v\|_{L^2(q_T)}+\|\rho_0(s)B\|_{L^2(Q_T)}\big)+\|  z_0\|_{H^1_0(\Omega)}\\
&\le Ce^{-\frac32s}(1+\|A\|_{L^\infty(Q_T)})\big(\|\rho_0(s)B\|_{L^2(Q_T)}+e^{cs}\|z_0\|_{H^1_0(\Omega)}\big)
\end{aligned}$$
that is (\ref{estimation1bis}).
\end{proof} 
  
\begin{remark}
 Remark that $c=\Vert \varphi(0,\cdot)\Vert_\infty>3/2$ so that the previous bound of $\Vert z\Vert_{L^\infty(Q_T)}$ is not uniform with respect to the parameter $s\geq 1$.
\end{remark}

\section{The least-squares method}\label{sec_LS}

In this section, we assume that the nonlinear function $g$ satisfies the hypothesis \ref{constraint_g_holder} for some $p\in [0,1]$ and that 
\begin{enumerate}[label=$\bf (H_2)$,leftmargin=1.5cm]
\item\label{growth_condition}\ There exists $\alpha\geq 0$ and $\beta>0$ such that $\vert g^{\prime}(r)\vert\leq \alpha + \beta \ln^{3/2}(1+\vert r\vert )$ for every $r$ in $\R$. 
\end{enumerate}
We introduce the notation 
\begin{equation}\label{estimg'}
\psi(r):=\alpha+\beta\ln^{3/2}(1+|r|), \quad \forall r\in \mathbb{R}. 
\end{equation}
We also assume that $g(0)=0$ leading in particular to the estimate $|g(r)|\le  |r|(\alpha+\beta\ln^{3/2}(1+|r|))$ for every $r\in\R$. The case $p=0$ corresponds to $\beta=0$ and $\alpha =\|g'\|_{L^\infty(\R)}$ and thus $\psi(r)\leq \|g'\|_{L^\infty(\R)}$ for every $r\in\R$.
  Remark that \ref{growth_condition} implies (\ref{cond_f_5}) and \ref{asymptotic_g}.

\subsection{The least-squares method}

We introduce, for all $s\ge s_0$,  the vector space $\mathcal{A}_0(s)$
\begin{equation}
\nonumber
\begin{aligned}
\mathcal{A}_0(s):=
\biggl\{(y,f): \ &\rho(s)\, y\in L^2(Q_T),  \  \rho_0(s) f\in L^2(q_T), \\
& \rho_0(s)(\partial_ty - \partial_{xx} y)\in L^2(Q_T),\  y(\cdot,0)=0\ \textrm{in}\ \Omega,\ y=0 \ \textrm{on}\  \Sigma_T\biggr\}
\end{aligned}
\end{equation}
where $\rho(s)$, $\rho_1(s)$ and $\rho_0(s)$ are defined in \eqref{def_weight}. $\mathcal{A}_0(s)$ endowed with the following scalar product
$$
\begin{aligned}
\big((y,f),(\overline{y},\overline{f})\big)_{\mathcal{A}_0(s)}:=\big(\rho(s) y,\rho(s) \overline{y}\big)_{2,q_T}
& + \big(\rho_0(s)f,\rho_0(s) \overline{f}\big)_{2,q_T}\\
&+\big(\rho_0(s) (\partial_ty-\partial_{xx} y),\rho_0(s)(\partial_t\overline{y}-\partial_{xx} \overline{y})\big)_{2}
\end{aligned}
$$
is a Hilbert space. The corresponding norm is $\Vert (y,f)\Vert_{\mathcal{A}_0(s)}=\sqrt{((y,f),(y,f))_{\mathcal{A}_0(s)}}$. We also consider the convex set 
\begin{equation}
\nonumber
\begin{aligned}
\mathcal{A}(s):=
\biggl\{(y,f):\ & \rho(s)\, y\in L^2(Q_T), \  \rho_0(s) f\in L^2(q_T), \\
& \rho_0(s)(\partial_ty - \partial_{xx} y)\in L^2(Q_T),\  y(\cdot,0)=u_0\ \textrm{in}\ \Omega,\ y=0 \ \textrm{on}\ \Sigma_T\biggr\}
\end{aligned}
\end{equation}
so that we can write $\mathcal{A}(s)=(\overline{y},\overline{f})+\mathcal{A}_0(s)$ for any element $(\overline{y},\overline{f})\in \mathcal{A}(s)$. We endow $\mathcal{A}(s)$ with the same norm. Clearly, if $(y,f)\in \mathcal{A}(s)$, then  $y\in C([0,T];L^2(\Omega))$ and since $\rho(s)\, y\in L^2(Q_T)$, then $y(\cdot,T)=0$. The null controllability requirement is therefore incorporated in the spaces $\mathcal{A}_0(s)$ and $\mathcal{A}(s)$.

\begin{remark}\label{remark2}
For any $(y,f)\in\mathcal{A}(s)$, since $\rho_0(s)\ge 1$ (see Remark \ref{remark-rho0}), we get that $\partial_ty - \partial_{xx} y\in L^2(Q_T)$; since $u_0\in H^1_0(\Omega)$,  standard arguments imply that $y\in L^\infty(Q_T)$ with
$$
\|y\|_{L^\infty(Q_T)} \le C \big(\|u_0\|_{H^1_0(\Omega)}+\|\partial_ty - \partial_{xx} y\|_{L^2(Q_T)} \big)
$$
In particular, for any $(y,f)\in\mathcal{A}_0(s)$,  $\|y\|_{L^\infty(Q_T)}\le Ce^{-\frac32s}\|(y,f)\|_{\mathcal{A}_0(s)}$ for some $C$ independent of $s$.
\end{remark}
\par\noindent
For any fixed $(\overline{y},\overline{f})\in \mathcal{A}(s)$ and $s\geq 0$, we can now consider the following non convex extremal problem :
\begin{equation}
\label{extremal_problem}
\min_{(y,f)\in \mathcal{A}_0(s)} E(s,\overline{y}+y,f+\overline{f})
\end{equation}
where the least-squares functional $E(s):\mathcal{A}(s)\to \mathbb{R}$ is defined as follows 
\begin{equation}
E(s,y,f):=\frac{1}{2}\biggl\Vert \rho_0(s)\biggl(\partial_ty-\partial_{xx} y + g(y)-f\,1_{\omega} \biggr)\biggr\Vert^2_{L^2(Q_T)}.
\end{equation}
We check that $\rho_0(s)g(y)\in L^2(Q_T)$ for any $(y,f)\in \mathcal{A}(s)$ so that $E(s)$ is well-defined. Precisely, using that $|g(r)|\le|r|\big(\alpha+\beta \ln^{3/2}(1+|r|)\big)=|r|\psi(r)$ for every $r$ and that $\rho_0\leq \rho$, we write
\begin{equation}\label{estim-rho0g}
\begin{aligned}
\|\rho_0(s)g(y)\|_{L^2(Q_T)}&\le  \Vert\rho_0(s)|y| \psi(y)\Vert_{L^2(Q_T)}\\
& \le \psi\big(\|y\|_{L^\infty(Q_T)}\big)\|\rho(s)y\|_{L^2(Q_T)}\leq \psi\big(\|y\|_{L^\infty(Q_T)}\big)\|(y,f)\|_{\mathcal{A}(s)}.
\end{aligned}
\end{equation}

Any pair $(y,f)\in \mathcal{A}$ for which $E(y,f)$ vanishes is a controlled pair of (\ref{heat-NL}), and conversely. In this sense, the functional $E$ is a so-called error functional which measures the deviation of $(y,f)$ from being a solution of the underlying nonlinear equation. 
Moreover, although the hypothesis \ref{growth_condition} is stronger \ref{asymptotic_g}, Theorem \ref{nullcontrollheatplus} proved in \cite{EFC-EZ} does not imply the existence of zero of $E$ in $\mathcal{A}(s)$, since controls of minimal $L^\infty(q_T)$ norm are considered in \cite{EFC-EZ}. Nevertheless, our constructive approach will show that, for $s$ large enough, the extremal problem (\ref{extremal_problem})
admits solutions $(y,f)\in \mathcal{A}(s)$ for which $E$ vanishes. 

We also emphasize that the $L^2(Q_T)$ norm in $E$ indicates that we are looking for regular weak solutions of the parabolic equation (\ref{heat-NL}). We refer to \cite{lemoine_gayte_munch} devoted to the case $g^\prime\in L^\infty(\mathbb{R})$ and the multidimensional case where the $L^2(0,T;H^{-1}(\Omega))$ is considered leading to weaker solutions. 

A practical way of taking a functional to its minimum is through some use of its derivative. In doing so, the presence of local minima is always something that  may dramatically spoil the whole scheme. The unique structural property that discards this possibility is the convexity of the functional $E$. However, for nonlinear equation like (\ref{heat-NL}), one cannot expect this property to hold for the functional $E$. Nevertheless, we are going to construct a minimizing sequence which always convergence to a zero of $E$. To do so, we introduce the following definition. 

\begin{defi}\label{def_Y1F1}
For any $s$ large enough and $(y,f)\in \mathcal{A}(s)$, we define the unique pair $(Y^1,F^1)\in \mathcal{A}_0(s)$ solution 
of 
\begin{equation}
\label{heat-Y1}
\left\{
\begin{aligned}
& \partial_tY^1 - \partial_{xx} Y^1 +  g^{\prime}(y) Y^1 = F^1 1_{\omega}+\partial_ty-\partial_{xx} y + g(y)-f\, 1_{\omega} \quad \textrm{in}\quad Q_T,\\
& Y^1=0 \,\, \textrm{on}\,\, \Sigma_T, \quad Y^1(\cdot,0)=0 \,\, \textrm{in}\,\, \Omega
\end{aligned}
\right.
\end{equation}
and which minimizes the functional $J$ defined in Theorem \ref{controllability_result}. In the sequel, it is called the minimal controlled pair. 
\end{defi}
The next proposition shows that there do exists some $(Y^1,F^1)$ in $\mathcal{A}_0(s)$. We emphasize that $F^1$ is a null control for the solution $Y^1$.  Preliminary, we prove the following result. 

\begin{lemma}\label{existence_triplet}
There exists $(y,f)\in L^2(Q_T)\times L^2(q_T)$ such $(y,f)\in \mathcal{A}(s)$ for all $s\geq 0$. 
\end{lemma}
\begin{proof}
Let $y^\star$ be the solution of
\begin{equation}
\nonumber
\left\{
\begin{aligned}
& \partial_t y^\star - \partial_{xx} y^\star= 0 \quad  \textrm{in}\quad Q_T,\\
& y^\star=0 \,\, \textrm{on}\,\, \Sigma_T, \quad y^\star(\cdot,0)=u_0\in H_0^1(\Omega) \,\, \textrm{in}\,\, \Omega,
\end{aligned}
\right.
\end{equation}
so that $y^\star\in L^2(0,T;H^2(0,1))$ and $\partial_t y^\star\in L^2(0,T;L^2(0,1))$. Let now any function $\phi\in C^\infty([0,T])$, $0\leq \phi\leq 1$ such that $\phi(0)=1$ and $\phi\equiv 0$ in $[T/2,T]$. Then, we easily check that the pair $(y,0)$ with $y:=\phi\, y^\star$ belongs to $\mathcal{A}(s)$ for any $s\geq 0$. 

Moreover, $\Vert y\Vert_{L^\infty(Q_T)}\leq  \Vert y^\star\Vert_{L^\infty(Q_T)}\leq C \Vert u_0\Vert_{H_0^1(\Omega)}$ so that any $s\geq \max(\Vert g^\prime\Vert^{2/3}_{L^\infty(0,C\Vert u_0\Vert_{H_0^1(\Omega)})},s_0)$ satisfies $s\geq \max(\Vert g^{\prime}(y)\Vert^{2/3}_{L^\infty(Q_T)},s_0)$.
\end{proof}

\begin{prop}  \label{proposition3} Let  $(y,f)\in \mathcal{A}(s)$ with $s \ge \max\big( \|g'(y)\|_{L^\infty(Q_T)}^{2/3},s_0\big)$. There exists a minimal controlled pair $(Y^1,F^1)\in \mathcal{A}_0(s)$ solution of (\ref{heat-Y1}). It satisfies the estimate: 
\begin{equation} \label{estimateF1Y1}
\Vert ( Y^1,F^1)\Vert_{\mathcal{A}_0(s)} \leq C \sqrt{E(s,y,f)} 
\end{equation}
for some $C >0$.
\end{prop}
\begin{proof}   For all $(y,f)\in \mathcal{A}(s)$, 
 $\rho_0(s)(\partial_ty-\partial_{xx} y+ g(y)-f1_\omega)\in L^2(Q_T)$. The existence of a null control $F^1$ is therefore given by Proposition \ref{controllability_result}. Choosing the control $F^1$ which minimizes together with the corresponding solution $Y^1$ the functional $J$ defined in Theorem \ref{controllability_result}, we get from (\ref{estimation1})-(\ref{estimation2}) the following estimates  (since $Y^1(\cdot,0)=0$) :
\begin{equation}
\label{estimateF1Y1rho}
\begin{aligned}
\Vert \rho(s)\, Y^1\Vert_{L^2(Q_T)}+\Vert \rho_0(s) F^1\Vert_{L^2(q_T)} &\leq C s^{-3/2} \Vert \rho_0(s) (\partial_ty-\partial_{xx} y+g(y)-f1_\omega)\Vert_{L^2(Q_T)}\\
&\leq C  s^{-3/2}\sqrt{E(s,y,f)}.
\end{aligned}
\end{equation}
%
Eventually, from the equation solved by $Y^1$, 
$$
\begin{aligned}
\Vert\rho_0(s)(\partial_tY^1-\partial_{xx}& Y^1)\Vert_{L^2(Q_T)} \\
&\leq \Vert \rho_0(s) F^1\Vert_{L^2(q_T)}+\Vert\rho_0(s) g^{\prime}(y)Y^1\Vert_{L^2(Q_T)}+ \Vert \rho_0(s)(\partial_ty-\partial_{xx} y+g(y)-f\, 1_{\omega})\Vert_{L^2(Q_T)}\\
&\leq \Vert \rho_0(s) F^1\Vert_{L^2(q_T)}+\Vert\rho_0(s) g^{\prime}(y)Y^1\Vert_{L^2(Q_T)}+ \sqrt{2E(s,y,f)}.
\end{aligned}
$$
But, since $\rho_0(s)\le \rho(s)$, using (\ref{estimateF1Y1rho}), we have
\begin{equation}\label{estimaterhogY1}
\|\rho_0(s)g'(y)Y^1\|^2_{L^2(Q_T)}
\le \|g'(y)\|^2_{L^\infty(Q_T)}\|\rho(s)Y^1\|^2_{L^2(Q_T)}\\
\le C s^{-3}  \|g'(y)\|^2_{L^\infty(Q_T)} E(s,y,f)
\end{equation}
thus
\begin{equation}
\label{boundY1}
\Vert\rho_0(s)(\partial_tY^1-\partial_{xx} Y^1-F^1\, 1_{\omega})\Vert_{L^2(Q_T)}\le  C\big(1+ s^{-3/2}\|g'(y)\|_{L^\infty(Q_T)}\big) \sqrt{E(s,y,f)}
\end{equation}
which proves that $(Y^1,F^1)$  belongs to $\mathcal{A}_0(s)$. Eventually, 
\begin{equation}
\nonumber
\begin{aligned}
\Vert ( Y^1,F^1)\Vert^2_{\mathcal{A}_0(s)}&=\Vert \rho(s)Y^1\Vert_2^2+ \Vert \rho_0(s)F^1\Vert_2^2+ \Vert \rho_0(s)(\partial_tY^1-\partial_{xx}Y^1)\Vert_2^2\\
&\leq \Vert \rho(s)Y^1\Vert_2^2+4\Vert \rho_0(s)F^1\Vert_2^2+3\Vert\rho_0(s) g^{\prime}(y)Y^1\Vert^2_{L^2(Q_T)}+ 3(\sqrt{2E(s,y,f)})^2\\
&\leq 4C s^{-3}E(s,y,f)+C s^{-3}  \|g'(y)\|^2_{L^\infty(Q_T)} E(s,y,f)+ 6 E(s,y,f)\\
&\leq C E(s,y,f)(1+ s^{-3}+ s^{-3} \|g'(y)\|^2_{L^\infty(Q_T)}).
\end{aligned}
\end{equation}
Since $s\ge \max(\|g'(y)\|_{L^\infty(Q_T)}^{2/3},s_0)\ge 1 $, we get $s^{-3}\leq 1$ and $s^{-3} \|g'(y)\|^2_{L^\infty(Q_T)}\leq 1$ leading to the result.
\end{proof}

\begin{remark}
From (\ref{heat-Y1}), we observe that $z:=y-Y^1\in L^2(\rho(s),Q_T)$ is a null controlled solution satisfying 
\begin{equation}
\label{heat-Y1-bis}
\left\{
\begin{aligned}
& \partial_tz - \partial_{xx} z +  g^{\prime}(y) z = (f-F^1) 1_{\omega} +g^{\prime}(y)y-g(y) \quad \textrm{in}\quad Q_T,\\
& z=0 \,\, \textrm{on}\,\, \Sigma_T, \quad z(\cdot,0)=u_0 \,\, \textrm{in}\,\, \Omega
\end{aligned}
\right.
\end{equation}
by the control $(f-F^1)\in L^2(\rho_0(s), q_T)$.
\end{remark}

\begin{remark}
We emphasize that the presence of a right hand side term in (\ref{heat-Y1}), namely $\partial_ty-\partial_{xx} y+g(y)-f\,1_{\omega}$, forces us to introduce the non trivial weights $\rho_0(s)$, $\rho_1(s)$   and $\rho(s)$ in the space $\mathcal{A}(s)$. This can be seen in the equality \eqref{solution-p1}:  since $\rho_0^{-1}(s) q$ belongs to $L^2(Q_T)$ for all $q\in P$, we need to impose that $\rho_0(s) B \in L^2(Q_T)$ with here $B= \partial_ty-\partial_{xx} y+g(y)-f\,1_{\omega}$. Working with the linearized equation \eqref{NL_z} (introduced in \cite{EFC-EZ}) which does not make appear any right hand side, we may avoid the introduction of Carleman type weights. Actually, \cite{EFC-EZ} considers controls of minimal $L^\infty(q_T)$ norm. 
Introduction of weights allows however the characterization \eqref{solution-p1}, which is very convenient at the practical level. We refer to \cite{EFC-MUNCH-SEMA} where this is discussed at length. 

We also emphasize that we have considered bounded weights at the initial time $t=0$ because of the constraints ``$\rho(s) y\in L^2(Q_T)$" and ``$y(0)=u_0$ in $\Omega$" 
appearing in the set $\mathcal{A}(s)$.

\end{remark}

\subsection{Main properties of the functional $E$}

The interest of the minimal controlled pair $(Y^1,F^1)\in \mathcal{A}_0(s)$ lies in the following result.

\begin{prop}\label{proposition4}
For any  $(y,f)\in \mathcal{A}(s)$ and $s\ge \max\big( \|g'(y)\|_{L^\infty(Q_T)}^{2/3},s_0\big)$, let $(Y^1,F^1)\in \mathcal{A}_0(s)$ defined in Definition \ref{def_Y1F1}. Then the derivative of $E(s)$ at the point $(y,f)\in \mathcal{A}(s)$ along the direction $(Y^1,F^1)$ given by $E^{\prime}(s,y,f)\cdot (Y^1,F^1):=\lim_{\eta\to 0,\eta\neq 0} \frac{E(s,(y,f)+\eta (Y^1,F^1))-E(s,y,f)}{\eta}$ satisfies 
\begin{equation}\label{estimateEEprime}
E^{\prime}(s,y,f)\cdot (Y^1,F^1)=2E(s,y,f).
\end{equation}
\end{prop}

\begin{proof}  We preliminary check that for all $(Y,F)\in\mathcal{A}_0(s)$,  $E(s)$  is differentiable at the point $(y,f)\in \mathcal{A}(s)$ along the direction $(Y,F)\in \mathcal{A}_0(s)$. For all $\lambda\in\R$, simple computations lead to the equality 
$$
\begin{aligned}
E(s,y+\lambda Y,f+\lambda F) =E(s,y,f)+ \lambda  E^{\prime}(s,y,f)\cdot (Y,F) + h\big(s,(y,f),\lambda (Y,F)\big)
\end{aligned}
$$
with 
\begin{equation}\label{Efirst}
E^{\prime}(s,y,f)\cdot (Y,F)=\biggl(\rho_0(s) (\partial_ty-\partial_{xx} y+ g(y)-f\, 1_\omega),   \rho_0(s) (\partial_tY-\partial_{xx} Y+ g^\prime(y)Y-F\, 1_\omega)\biggr)_{L^2(Q_T)}
\end{equation}
and
$$
\begin{aligned}
h(s,(y,f),\lambda (Y,F)):=& \lambda \biggl(\rho_0(s) (\partial_tY-\partial_{xx} Y+ g^\prime(y)Y-F\, 1_\omega),\rho_0(s) l(y,\lambda Y)\biggl)_{L^2(Q_T)}\\
& +\frac{\lambda ^2}2\| \rho_0(s) (\partial_tY-\partial_{xx} Y+ g^\prime(y)Y-F\, 1_\omega)\|_{L^2(Q_T)}^2\\
& +\biggl(\rho_0(s) (\partial_ty-\partial_{xx} y+ g(y)-f\, 1_\omega),\rho_0(s) l(y,\lambda Y)\biggr)_{L^2(Q_T)}\\
&+ \frac{1}{2}\|\rho_0(s) l(y,\lambda Y)\|_{L^2(Q_T)}^2
\end{aligned}
$$ 
where $l(y,\lambda Y):=g(y+\lambda Y)-g(y)-\lambda g^{\prime}(y)Y$.

The application $(Y,F)\to E^{\prime}(s,y,f)\cdot (Y,F)$ is linear and continuous from $\mathcal{A}_0(s)$ to $\mathbb{R}$ as it satisfies using  (\ref{estimateF1Y1}),  (\ref{estimateF1Y1rho}) and (\ref{estimaterhogY1}) :
\begin{equation}
\begin{aligned}\label{45}
\vert E^{\prime}&(s,y,f)\cdot (Y,F)\vert \\
& \leq \Vert \rho_0(s) (\partial_ty-\partial_{xx} y+ g(y)-f\, 1_\omega)\Vert_{L^2(Q_T)} \Vert \rho_0(s) (\partial_tY-\partial_{xx} Y+ g^\prime(y)Y-F\, 1_\omega)\Vert_{L^2(Q_T)}\\
& \leq \sqrt{2E(s,y,f)} \biggl(\Vert \rho_0(s)(\partial_tY-\partial_{xx} Y)\Vert_2+ \Vert\rho_0(s) F\Vert_{L^2(q_T)} + \Vert \rho_0(s) g^\prime(y)Y\Vert_{L^2(Q_T)} \biggr)\\
& \leq \sqrt{6}\big(1+ \|g'(y)\|_{L^\infty(Q_T)}\big) \sqrt{E(s,y,f)}  \Vert (Y,F)\Vert_{\mathcal{A}_0(s)}.  
\end{aligned}
\end{equation}
Similarly, for all $\lambda \in\R^\star$
$$
\begin{aligned}
\biggl|\frac{1}{\lambda }h\big(s,(y,f),\lambda (Y,F)\big)\biggr| \leq
&  \biggl(|\lambda| \Vert \rho_0(s) (\partial_tY-\partial_{xx} Y+ g^\prime(y)Y-F\, 1_\omega)\Vert_{L^2(Q_T)} +\sqrt{2E(s,y,f)}\\
&\hskip 2cm +\frac12\Vert \rho_0(s) l(y,\lambda Y)\Vert_{L^2(Q_T)}\biggl)\frac1{|\lambda| }\Vert \rho_0(s) l(y,\lambda Y)\Vert_{L^2(Q_T)}\\
&+\frac{|\lambda| }2\Vert \rho_0(s) (\partial_tY-\partial_{xx} Y+ g^\prime(y)Y-F\, 1_\omega)\Vert_{L^2(Q_T)}^2.
\end{aligned}
$$
Since $g'\in\mathcal{C}(\R)$ we have, a.e in $Q_T$ :
$\Big|\frac1\lambda l(y,\lambda Y\Big|=\Big|\frac{g(y+\lambda Y)-g(y)}\lambda-g'(y)Y\Big|\to 0$ as $\lambda\to 0$
and, since $Y\in L^\infty(Q_T)$ and $y\in L^\infty(Q_T)$, a.e in $Q_T$
$$\Big|\frac1\lambda l(y,\lambda Y\Big|=\Big|\frac{g(y+\lambda Y)-g(y)}\lambda-g'(y)Y\Big|\le \big(\sup_{\theta\in[0,1]}\|g'(y+\theta Y)\|_{L^\infty(Q_T)}+\|g'(y)\|_{L^\infty(Q_T)}\big)|Y|$$
and therefore (recalling that $\rho_0\le \rho$) 
$$\begin{aligned}
\frac1{|\lambda| }\Vert \rho_0(s) l(y,\lambda Y)\Vert_{L^2(Q_T)}
&\le \big(\sup_{\theta\in[0,1]}\|g'(y+\theta Y)\|_{L^\infty(Q_T)}+\|g'(y)\|_{L^\infty(Q_T)}\big)\|\rho_0 Y\|_{L^2(Q_T)}\\
&\le \big(\sup_{\theta\in[0,1]}\|g'(y+\theta Y)\|_{L^\infty(Q_T)}+\|g'(y)\|_{L^\infty(Q_T)}\big)\|\rho Y\|_{L^2(Q_T)}.
\end{aligned}
$$
It then follows from the Lebesgue dominated convergence theorem that $\frac1{\lambda }\Vert \rho_0(s) l(y,\lambda Y)\Vert_{L^2(Q_T)}\to 0$ as $\lambda\to 0$ and therefore that $h(s,(y,f),\lambda (Y,F))=o(\lambda)$.
Thus the functional $E(s)$  is differentiable at the point $(y,f)\in \mathcal{A}(s)$ along the direction $(Y,F)\in\mathcal{A}_0(s)$.
Eventually, the equality (\ref{estimateEEprime}) follows from the definition of the pair $(Y^1,F^1)$ given in (\ref{heat-Y1}).
\end{proof} 

\vskip 0.25cm
Remark that from the equality \eqref{Efirst}, the derivative $E^{\prime}(s,y,f)$ is independent of $(Y,F)$. We can then define the norm $\Vert E^{\prime}(s,y,f)\Vert_{(\mathcal{A}_0(s))^{\prime}}:= \sup_{(Y,F)\in \mathcal{A}_0(s), (Y,F)\neq (0,0)} \frac{E^{\prime}(s,y,f)\cdot (Y,F)}{\Vert (Y,F)\Vert_{\mathcal{A}_0(s)}}$ associated to $\mathcal{A}_0^\prime(s)$, the set of the linear and continuous applications from $\mathcal{A}_0(s)$ to $\mathbb{R}$.

\vskip 0.25cm

Combining the equality \eqref{estimateEEprime} and the inequality \eqref{estimateF1Y1}, we deduce the following estimates of $E(s,y,f)$ in term of the norm of $E^\prime(s,y,f)$.

\begin{prop}
For any  $(y,f)\in \mathcal{A}(s)$ and $s\ge \max\big( \|g'(y)\|_{L^\infty(Q_T)}^{2/3},s_0\big)$, the inequalities hold  true
$$
\frac1{\sqrt{6}\big(1+ \|g'(y)\|_{L^\infty(Q_T)}\big) } \|E'(s,y,f)\|_{ \mathcal{A}_0^\prime(s)}\le \sqrt{E(s,y,f)}\le C  \|E'(s,y,f)\|_{ \mathcal{A}_0^\prime(s)}
$$
where $C>0$ is the constant appearing in Proposition \ref{proposition3}.
\end{prop}

\begin{proof}  \eqref{estimateEEprime} rewrites
$E(s,y,f)=\frac12 E^{\prime}(s,y,f)\cdot (Y^1,F^1)$
where $(Y^1,F^1)\in \mathcal{A}_0(s)$ is solution of (\ref{heat-Y1})  and therefore, with \eqref{estimateF1Y1}
$$\begin{aligned}
E(s,y,f)
&\le\frac12 \|E'(s,y,f)\|_{ \mathcal{A}_0^\prime(s)} \|(Y^1,F^1)\|_{ \mathcal{A}_0(s)}\\
&\le C \|E'(s,y,f)\|_{ \mathcal{A}_0^\prime(s)}\sqrt{E(s,y,f)}.
\end{aligned}
$$
On the other hand, using (\ref{45}), for all $(Y,F)\in\mathcal{A}_0(s)$   :
$$
\vert E^{\prime}(s,y,f)\cdot (Y,F)\vert \leq  \sqrt{6}\big(1+\|g'(y)\|_{L^\infty(Q_T)}\big)   \sqrt{E(s,y,f)} \Vert (Y,F)\Vert_{\mathcal{A}_0(s)}$$
leading to the left inequality. 
\end{proof} 

In particular, any \textit{critical} point $(y,f)\in \mathcal{A}(s)$ for $E(s)$ (i.e. for which $E^\prime(s,y,f)$ vanishes) is a zero for $E(s)$, a pair solution of the controllability problem. 
In other words, any sequence $(y_k,f_k)_{k\in \mathbb{N}}$ of $\mathcal{A}(s)$ satisfying $\Vert E^\prime(s,y_k,f_k)\Vert_{\mathcal{A}_0^\prime(s)}\to 0$ as $k\to \infty$ and for which $(\Vert g^\prime(y_k)\Vert_\infty)_{k\in \mathbb{N}}$ is bounded
is such that $E(s,y_k,f_k)\to 0$ as $k\to \infty$. We insist that this property does not imply the convexity of the functional $E(s)$ (nor \textit{a fortiori} the strict convexity of $E(s)$, which actually does not hold here in view of the multiple zeros for $E(s)$) but show that a minimizing sequence for $E(s)$ can not be stuck in a local minimum. Our least-squares algorithm, designed in the next section, is based on that property. 

Eventually, the left inequality indicates that the functional $E(s)$ is flat around its zero set. As a consequence, gradient based minimizing sequences for $E(s)$ are inefficient as they usually achieve a low rate of convergence (we refer to \cite{AM-PP-2014}
and also \cite{lemoine-Munch-Pedregal-AMO-20} devoted to the Navier-Stokes equation where this phenomenon is observed).

We end this section with the following crucial estimate.

\begin{lemma}\label{estimW1lemma}
Assume that $g$ satisfies \ref{constraint_g_holder} for some $p\in [0,1]$. Let  $(y,f)\in \mathcal{A}(s)$,  $s\ge \max\big( \|g'(y)\|_{L^\infty(Q_T)}^{2/3},s_0\big)$,   and $(Y^1,F^1)\in \mathcal{A}_0(s)$ given in Definition \ref{def_Y1F1} associated with $(y,f)$. For any $\lambda\in \mathbb{R}_+$ the following estimate holds 
\begin{equation} 
\sqrt{E\big(s,(y,f)-\lambda (Y^1,F^1)\big)}  \leq    \sqrt{E(s,y,f)} \biggl(\vert 1-\lambda\vert  +\lambda^{p+1}c_1(s) \sqrt{E(s,y,f)}^p\biggr) \label{estimW1}
\end{equation}
with
\begin{equation}
\label{c1}
c_1(s) := \frac{C^{1+p}}{1+p}s^{-3/2} e^{-\frac{3p}2 s} [g']_p.
\end{equation}
\end{lemma} 

\begin{proof}  
For any $(x,y)\in \R^2$, $y\not=0$,  and $\lambda\in \R$, we write $g(x+\lambda y)-g(x)=\int_0^\lambda y g'(x+\xi y)d\xi$ leading to 
$$
\begin{aligned}
|g(x+\lambda y)-g(x)-\lambda g'(x)y|
&\le \int_0^\lambda |y| |g'(x+\xi y)-g'(x)|d\xi\\
&\le \int_0^\lambda |y|^{1+p}|\xi|^p\frac{|g'(x+\xi y)-g'(x)|}{|\xi y|^p}d\xi\\
&\le  [g']_p |y|^{1+p}\frac{|\lambda|^{1+p}}{1+p}.
\end{aligned}
$$
It follows that 
$$\rho_0(s)|\frac{1}{\lambda }l(y,\lambda Y^1)|=\rho_0(s)\Big|\frac{g(y+\lambda Y^1)-g(y)}{\lambda }-g'(y)Y^1 \Big |\le 
  [g']_p\rho_0(s)|Y^1|^{1+p}\frac{|\lambda|^{p}}{1+p}
 $$
 and thus, since $Y^1\in L^\infty(Q_T)$ (see Remark \ref{remark2}) and $\rho_0(s)\le\rho(s)$ :
\begin{equation}\label{38}\begin{aligned}
\frac1{|\lambda| }\Vert \rho_0(s) l(y,\lambda Y^1)\Vert_{L^2(Q_T)}
& \le 
  [g']_p \frac{|\lambda|^{p}}{1+p}\Vert\rho_0(s)|Y^1|^{1+p}\Vert_{L^2(Q_T)}\\
 & \le  [g']_p \frac{|\lambda|^{p}}{1+p}\Vert Y^1\Vert_{L^\infty(Q_T)}^p \Vert\rho(s) Y^1\Vert_{L^2(Q_T)}\\
\end{aligned}\end{equation}
and obtain that
\begin{equation}\label{E_expansionb}
\begin{aligned}
&2 E\big(s,(y,f)-\lambda (Y^1,F^1)\big) \\
&=\biggl\Vert \rho_0(s)\big(\partial_ty-\partial_{xx} y+g(y)-f\,1_\omega\big)-
\lambda \rho_0(s) \big(\partial_tY^1-\partial_{xx} Y^1+g^\prime(y)Y^1-F\,1_\omega\big)+\rho_0(s)l(y,-\lambda Y^1)\biggr\Vert^2_{L^2(Q_T)}\\
& =\biggl\Vert \rho_0(s)(1-\lambda)\big(\partial_ty-\partial_{xx} y+g(y)-f\,1_\omega\big)+\rho_0(s)l(y,-\lambda Y^1)\biggr\Vert^2_{L^2(Q_T)}\\
&\le \Big(\bigl\Vert \rho_0(s)(1-\lambda)\big(\partial_ty-\partial_{xx} y+g(y)-f\,1_\omega\big)\bigr \Vert_{L^2(Q_T)} +\bigl \Vert\rho_0(s)l(y,-\lambda Y^1)\bigr\Vert_{L^2(Q_T)}\Big)^2\\
&\leq 2 \biggl( \vert 1-\lambda\vert \sqrt{E(s,y,f)}+\frac{\lambda^{p+1}}{\sqrt{2}(p+1)}  [g']_p\Vert Y^1\Vert_{L^\infty(Q_T)}^p \Vert\rho(s) Y^1\Vert_{L^2(Q_T)}\biggr)^2.
\end{aligned}
\end{equation}
Using Remark \ref{remark2} and estimates (\ref{estimateF1Y1})  and (\ref{estimateF1Y1rho}), we obtain
$$
\begin{aligned}
2 E\big(s,(y,f)&-\lambda (Y^1,F^1)\big)  
\leq 2 \biggl( \vert 1-\lambda\vert \sqrt{E(s,y,f)}+Ce^{-\frac{3p}2 s}\frac{\lambda^{p+1}}{p+1}  [g']_p\|(Y^1,F^1)\|_{\mathcal{A}_0(s)}^{p}s^{-3/2}\sqrt{E(s,y,f)}\biggr)^2\\
 &\leq 2 \biggl( \vert 1-\lambda\vert \sqrt{E(s,y,f)}+\lambda^{p+1}\frac{Cs^{-3/2}e^{-\frac{3p}2 s}C^p}{p+1}  [g']_pE(s,y,f)^{\frac{p+1}2}\biggr)^2
\end{aligned}
$$
from which we get \eqref{estimW1}.
\end{proof}

\section{Convergence of the least-squares method}\label{sec_convergence}

 We now examine the convergence of an appropriate sequence $(y_k,f_k)\in \mathcal{A}(s)$. In this respect, we observe from the equality (\ref{estimateEEprime}) that, for any $(y,f)\in \mathcal{A}(s)$, $-(Y^1,F^1)$ given in Definition \ref{def_Y1F1}, is a descent direction for the functional $E(s)$ at the point $(y,f)$, as soon as $s\ge  \max( \|g'(y)\|_{L^\infty(Q_T)}^{2/3},s_0)$. Therefore, we can define at least formally, for any fixed $m\geq 1$, a minimizing sequence $(y_k,f_k)_{k\in\mathbb{N}}\in \mathcal{A}(s)$ as follows: 
\begin{equation}
\label{algo_LS_Y}
\left\{
\begin{aligned}
&(y_0,f_0) \in \mathcal{A}(s), \\
&(y_{k+1},f_{k+1})=(y_k,f_k)-\lambda_k (Y^1_k,F_k^1), \quad k\ge 0, \\
& \lambda_k= \textrm{argmin}_{\lambda\in[0,m]} E\big(s,(y_k,f_k)-\lambda (Y^1_k,F_k^1)\big)    
\end{aligned}
\right.
\end{equation}
where $(Y^1_k,F_k^1)\in \mathcal{A}_0(s)$ is the minimal controlled pair solution of 
\begin{equation}
\label{heat-Y1k}
\left\{
\begin{aligned}
& \partial_tY^1_{k} - \partial_{xx} Y^1_k +  g^{\prime}(y_k) Y^1_k = F^1_k 1_{\omega}+ \partial_ty_{k}-\partial_{xx} y_k+g(y_k)-f_k 1_\omega \quad \textrm{in}\quad Q_T,\\
& Y_k^1=0 \,\,\, \textrm{on}\,\,\, \Sigma_T, \quad Y_k^1(\cdot,0)=0 \,\,\, \textrm{in}\,\,\, \Omega
\end{aligned}
\right.
\end{equation}
associated with $(y_k,f_k)\in \mathcal{A}(s)$. In particular, the pair $Y^1_k,F^1_k$ vanishes when $E(s,y_k,f_k)$ vanishes. 
The real number $m\geq 1$ is arbitrarily fixed and is introduced in order to keep the sequence $(\lambda_k)_{k\in\mathbb{N}}$ bounded.

We highlight that, in order to give a meaning to (\ref{algo_LS_Y}), we need to prove that we can choose  the parameter $s$ \underline{independent of $k$}, that is  $s\ge \max\big( \|g'(y_k)\|_{L^\infty(Q_T)}^{2/3},s_0\big)$ for all $k\in\N$. In this respect, it suffices to prove that there exists $M>0$ such that  $\|y_k\|_{L^\infty(Q_T)}\le M$ for every $k\in\N$. Under \ref{growth_condition}, this implies that  $\|g'(y_k)\|_{L^\infty(Q_T)}\le \psi(M)$ for every $k\in \mathbb{N}$, where $\psi$ is defined in (\ref{estimg'}). We shall prove the existence of such $M$ by induction. 


\begin{prop}\label{Estimation yk+1} 
Assume that $g$ satisfies \ref{growth_condition} and \ref{constraint_g_holder} for some $p\in [0,1]$. Let $M>0$ large enough and $s\ge \max\big(C(p)\psi(M)^{2/3},s_0\big)$ with $C(p)=1$ if $p\in(0,1]$, and $C(0)=(2C)^{3/2}$. Let $(y_0,f_0)\in \mathcal{A}(s)$ such that  $M\ge \|y_0\|_{L^\infty(Q_T)}$. Assume that, for some $n\geq 0$, $(y_k,f_k)_{0\leq k\leq n}$ defined from \eqref{algo_LS_Y}  satisfies $\|y_k\|_{L^\infty(Q_T)}\le M$. Then
\begin{equation}\label{equation44}
\begin{aligned}
\|y_{n+1}\|_{L^\infty(Q_T)}
 \le & \|y_{0}\|_{L^\infty(Q_T)}+Cm \max\biggl( \frac{p+1}{p}\sqrt{E(s,y_0,f_0)} , \frac{(1+p)^{\frac1p+1}}{p}  c_1^{1/p}(s){E(s,y_0,f_0)}\biggr)
 \end{aligned}
 \end{equation}
 if $p\in (0,1]$
 and 
\begin{equation}\label{equation44bis}
\|y_{n+1}\|_{L^\infty(Q_T)}
 \le  \|y_{0}\|_{L^\infty(Q_T)}+Cm\frac{\sqrt{E(s,y_0,f_0)}}{1-c_1(s)}
 \end{equation}
 if $p=0$.
\end{prop}
We point out that the existence of $(y_0,f_0)\in \A(s)$ follows from Lemma \ref{existence_triplet}.
\begin{proof} 
The inequality $\Vert y_n\Vert_{L^{\infty}(Q_T)}\leq M$ implies that $\Vert g^\prime(y_n)\Vert_{L^\infty(Q_T)}^{2/3}\leq \psi(M)^{2/3}$ and then 
\\ $s\ge \max\big(\psi(M)^{2/3},s_0\big)\geq \max\big( \|g'(y_n)\|_{L^\infty(Q_T)}^{2/3},s_0\big)$. Proposition \ref{proposition3} allows to construct the pair  sequence $(Y^1_n,F^1_n)\in \mathcal{A}_0(s)$ solution of (\ref{heat-Y1}). Then, (\ref{algo_LS_Y}) allows to define $(y_{n+1},f_{n+1})$.
Estimate (\ref{estimW1}) implies that 
\begin{equation} 
\nonumber
\sqrt{E\big(s,(y_k,f_k)-\lambda (Y_k^1,F_k^1)\big)}  \leq   \sqrt{E(s,y_k,f_k)} \biggl(\vert 1-\lambda\vert  +\lambda^{p+1}c_1(s) \sqrt{E(s,y_k,f_k)}^p\biggr) \end{equation}
and then 
\begin{equation}\label{estimW2}
\sqrt{E(s,y_{k+1},f_{k+1})}  \leq  \sqrt{E(s,y_k,f_k)} \min_{\lambda\in [0,m]}p_k(s,\lambda)
\end{equation}
where 
\begin{equation}\label{pk}
p_k(s,\lambda):= \vert 1-\lambda\vert + \lambda^{p+1}c_1(s)  E(s,y_k,f_k)^{p/2}, \quad \forall \lambda\in\mathbb{R}, \, \forall s>0.
\end{equation}
Since $\big(E(s,y_k,f_k)\big)_{0\le k\le n}$ decreases, $\big(p_k({s,\lambda})\big)_{0\le k\le n}$ decreases for all $\lambda$  ($p_k$ do not depend on $k$ if $p=0$) and thus, defining $p_k(s,\widetilde{\lambda_k}):=\min_{\lambda\in[0,m]}p_k(s,\lambda)$, $\big(p_k(s,\widetilde{\lambda_k})\big)_{0\le k\le n}$ decreases as well. \eqref{estimW2} then implies, for all $0\le k\le n-1$, that 
\begin{equation}
 \sqrt{E(s,y_{k+1},f_{k+1})} \leq   \sqrt{E(s,y_k,f_k)}  p_k(s,\widetilde{\lambda_k})
 \leq   \sqrt{E(s,y_k,f_k)}  p_0(s,\widetilde{\lambda_0}). \label{decreaseEk}
\end{equation}

\par\noindent
{\bf First case : $p\in(0,1]$}.
We prove that
\begin{equation}\label{52}
\sum_{k=0}^n \sqrt{E(s,y_k,f_k)}\leq\max\biggl( \frac{1}{p}\sqrt{E(s,y_0,f_0)} , \frac{(1+p)^{\frac1p+1}}{p}  c_2(s){E(s,y_0,f_0)}\biggr)
\end{equation}
where $c_2(s):=c_1^{1/p}(s)$ .
Since $p_0'(s,0)=-1$,  $p_0(s,\widetilde{\lambda_0})< p_0(s,0)=1$ we deduce from (\ref{decreaseEk}) that  :
\begin{equation}\label{sommeE}
\sum_{k=0}^n \sqrt{E(s,y_k,f_k)}\leq \sqrt{E(s,y_0,f_0)} \frac{1-p_0(s,\widetilde{\lambda_0})^{n+1}}{1-p_0(s,\widetilde{\lambda_0})}\leq \frac{\sqrt{E(s,y_0,f_0)}}{1-p_0(s,\widetilde{\lambda_0})}.
\end{equation}
If $c_2(s)\sqrt{E(s,y_0,f_0)}< \frac1{(p+1)^{1/p}}$, we check that $p_0(s,\widetilde{\lambda_0}) \le p_0(s,1)=c_1(s)\sqrt{E(s,y_0,f_0)}^p\leq \frac{1}{p+1}$ and thus
$$\frac{\sqrt{E(s,y_0,f_0)}}{1-p_0(s,\widetilde{\lambda_0})}\le \frac{p+1}{p}\sqrt{E(s,y_0,f_0)}.$$
If $c_2(s)\sqrt{E(s,y_0,f_0)}\ge  \frac1{(p+1)^{1/p}}$,  then for all $\lambda\in[0,1]$, 
$p_0'(s,\lambda)=-1+(p+1)\lambda^p c_1(s) E(s,y_0,f_0)^{p/2}$ and thus $p_0'(s,\lambda)=0$ if and only if $\lambda =\frac{1}{(p+1)^{1/p}c_2(s)\sqrt{E(s,y_0,f_0)}}$ leading to
$$p_0(s,\widetilde{\lambda_0})=1-\frac{p}{(1+p)^{\frac1p+1}}\frac{1}{c_2(s)\sqrt{E(s,y_0,f_0)}}$$
and 
$$
\frac{\sqrt{E(s,y_0,f_0)} }{1-p_0(s,\widetilde{\lambda_0})}\le \frac{(1+p)^{\frac1p+1}}{p} c_2(s){E(s,y_0,f_0)}. 
$$
\eqref{sommeE} then leads to (\ref{52}). Then \eqref{algo_LS_Y} implies that $y_{n+1}=y_0-\sum_{k=0}^n \lambda_k Y^1_k$ and thus, using  (\ref{estimateF1Y1})
$$\begin{aligned}
\|y_{n+1}\|_{L^\infty(Q_T)}
&\le \|y_{0}\|_{L^\infty(Q_T)}+m \sum_{k=0}^n\|Y^1_k\|_{L^\infty(Q_T)}\le \|y_{0}\|_{L^\infty(Q_T)}+Cm \sum_{k=0}^n\|(Y^1_k,F^1_k)\|_{\mathcal{A}_0(s)}\\
&\le \|y_{0}\|_{L^\infty(Q_T)}+C m \sum_{k=0}^n \sqrt{E(s,y_k,f_k)}
\end{aligned}
$$
which gives (\ref{equation44}), using (\ref{52}).

\par\noindent
{\bf Second case : $p=0$}. Recall that for $p=0$,  $\psi(r)=\Vert g^\prime\Vert_{\infty}$ for every $r\in \mathbb{R}$. Then simply $p_k(s,\lambda)=|1-\lambda|+\lambda c_1(s)$ for all $k$ and 
$$
p_0(s,\widetilde{\lambda_0})=\min_{\lambda\in[0,m]}p_0(s,\lambda)=\min_{\lambda\in[0,1]}p_0(s,\lambda)= p_0(s,1)=c_1(s)
$$
with $c_1(s)=Cs^{-3/2}  [g']_{0}=2 C s^{-3/2}\Vert g^\prime\Vert_\infty=2 C s^{-3/2}\psi(M)$. Taking $s$ large enough, precisely $s> \max((2C)^{1/3} \psi(M)^{2/3},s_0)$, we obtain that $c_1(s)<1$. We then have  
for all $0\le k\le n-1$ that 
$$
 \sqrt{E(s,y_{k+1},f_{k+1})} \leq    \sqrt{E(s,y_k,f_k)} \, c_1(s)$$
and thus
\begin{equation}\label{sommeEp0}
\sum_{k=0}^n \sqrt{E(s,y_k,f_k)}\leq \sqrt{E(s,y_0,f_0)} \frac{1-c_1(s)^{n+1}}{1-c_1(s)}\leq \frac{\sqrt{E(s,y_0,f_0)}}{1-c_1(s)}.
\end{equation}
Proceeding as before, we get (\ref{equation44bis}).
\end{proof}

In view of estimates (\ref{equation44}) and (\ref{equation44bis}), we now intend to choose $s$ such that $\|y_{n+1}\|_{L^\infty(Q_T)}\le M$. To this end, we need an estimate of $E(s,y_0,f_0)=\frac{1}{2} \| \rho_0(s) (\partial_t y_0-\partial_{xx} y_0 + g(y_0)-f_0\,1_{\omega}  )\|^2_{L^2(Q_T)}$ in terms of $s$. Since $\rho(s)\notin L^2(Q_T)$, such estimate is not straightforward for any $(y_0,f_0)\in \mathcal{A}(s)$. We select the pair $(y_0,f_0)\in \mathcal{A}(s)$ solution of the linear problem, i.e. $g\equiv 0$ in \eqref{heat-NL}.

\begin{lemma}\label{y0f0g0}
Assume that $g$ satisfies \ref{growth_condition}. For any $s\geq s_0$, let $(y_0,f_0)\in \mathcal{A}(s)$ be the solution of the extremal problem \eqref{extrema-problem} in the linear case for which $g\equiv 0$. Then, 
\begin{equation}\label{equation52}
\sqrt{E(s,y_0,f_0)} 
\le   \biggl(\alpha+\beta\big(c^{3/2}+ \ln^{3/2}(1+C\Vert u_0\Vert_{H_0^1(\Omega)})\big)\biggr)e^{cs }\|u_0\|_{H_0^1(\Omega)}.
\end{equation}
with $c=\Vert \varphi(\cdot,0)\Vert_\infty$.
\end{lemma}
\begin{proof}
Estimate \eqref{estimation1} of Proposition \ref{controllability_result} with $A=0,B=0$ and $z_0=u_0$ leads to 
\begin{equation}\label{estimationy0}
\Vert \rho(s) y_0\Vert_{L^2(Q_T)}+ \Vert \rho_0(s)f_0\Vert_{L^2(q_T)} \leq  Cs^{-3/2} e^{ cs }\|u_0\|_2
\end{equation}
while (\ref{estimation1bis}) leads to, since $s>1$ and $\rho_0\ge 1$,
\begin{equation}
\label{estimy0Linfinty}
\|y_0\|_{L^\infty(Q_T)}  \leq C e^{-\frac32s}e^{cs}\|u_0\|_{H_0^1(\Omega)}  \leq C e^{cs}\|u_0\|_{H_0^1(\Omega)}.
\end{equation}
It follows that, using  (\ref{estim-rho0g}), since $\rho_0\le \rho$ and $s\ge 1$ 
\begin{equation}
\nonumber
\begin{aligned}
\sqrt{E(s,y_0,f_0)}    
&=\frac1{\sqrt{2}}\|\rho_0(s) g(y_0)\|_{L^2(Q_T)}\le \frac1{\sqrt{2}}\psi (\|y_0\|_{L^\infty(Q_T)})\|\rho(s) y_0\|_{L^2(Q_T)}\\
&\le C\psi(C e^{cs}\|u_0\|_{H^1_0(\Omega)})s^{-3/2}e^{cs }\|u_0\|_{H_0^1(\Omega)}\\
&\le C\big(\alpha+\beta\ln^{3/2}(1+Ce^{sc}\Vert u_0\Vert_{H_0^1(\Omega)})\big)s^{-3/2}e^{cs }\|u_0\|_{H_0^1(\Omega)}\\
&\le C\big(\alpha+\beta\ln^{3/2}(e^{sc}(1+C\Vert u_0\Vert_{H_0^1(\Omega)}))\big)s^{-3/2}e^{cs }\|u_0\|_{H_0^1(\Omega)}\\
&\le C\biggl(\alpha+\beta\big((sc)^{3/2}+ \ln^{3/2}(1+C\Vert u_0\Vert_{H_0^1(\Omega)})\big)\biggr)s^{-3/2}e^{cs }\|u_0\|_{H_0^1(\Omega)}\\
&\le C\biggl(\alpha+\beta\big(c^{3/2}+ \ln^{3/2}(1+C\Vert u_0\Vert_{H_0^1(\Omega)})\big)\biggr)e^{cs }\|u_0\|_{H_0^1(\Omega)}.
\end{aligned}
\end{equation}
\end{proof}
We are now in position to prove to following result. 

\begin{prop}\label{Estimation-2-yk+1}
Assume that $g$ satisfies \ref{growth_condition} and \ref{constraint_g_holder} for some $p\in [0,1]$. Assume moreover that $2cC(p)\beta^{2/3}<1$
and let $(y_0,f_0)$ be the controlled pair given by Lemma \ref{y0f0g0}. 
There exists $M_0>0$ such that, if we have constructed from (\ref{algo_LS_Y}) the pairs $(y_k,f_k)_{0\le k\le n}\in \mathcal{A}(s)$ with $s=\max(C(p)\psi(M_0)^{2/3},s_0)$ satisfying $\|y_k\|_{L^\infty(Q_T)}\le M_0$ for all $0\le k\le n$, then the pair $(y_{n+1},f_{n+1})$ constructed from (\ref{algo_LS_Y}) also belongs to $\mathcal{A}(s)$ and satisfies 
$$
\|y_{n+1}\|_{L^\infty(Q_T)}
 \le M_0.
 $$
\end{prop}

\begin{proof} Assume that for some $M$ larhe enough, $\Vert y_{k}\Vert_{L^\infty(Q_T)}\leq M$. The inequality $(a+b)^{2/3}\leq a^{2/3}+b^{2/3}$  for all $a,b\geq 0$ allows to write 
$$
\psi(M)^{2/3}=(\alpha+\beta \ln^{3/2}(1+M))^{2/3}\leq \alpha^{2/3}+\beta^{2/3}\ln(1+M).
$$
Assume that for some $M$ large enough, $\Vert y_{k}\Vert_{L^\infty(Q_T)}\leq M$. Estimate \eqref{estimy0Linfinty} with $s=s=\max(C(p)\psi(M_0)^{2/3},s_0)$ then leads to 
\begin{equation}
\label{u0linfty}
\begin{aligned}
\Vert y_0\Vert_{L^\infty(Q_T)} & \leq C \Vert u_0\Vert_{H_0^1(\Omega)} e^{cs}\leq C \Vert u_0\Vert_{H_0^1(\Omega)} e^{c (s_0+C(p)\psi(M)^{2/3})}\leq 
Ce^{c s_0}\Vert u_0\Vert_{H_0^1(\Omega)}e^{c C(p)\psi(M)^{2/3}}\\
& \leq Ce^{c s_0}\Vert u_0\Vert_{H_0^1(\Omega)}e^{c C(p)\big(\alpha^{2/3}+\beta^{2/3}\ln(1+M)\big)}\\
& \leq Ce^{c s_0}\Vert u_0\Vert_{H_0^1(\Omega)}e^{C(p)\alpha^{2/3}} e^{c C(p)\beta^{2/3}\ln(1+M)}\\
& \leq Ce^{c s_0}\Vert u_0\Vert_{H_0^1(\Omega)}e^{C(p)\alpha^{2/3}} (1+M)^{cC(p)\beta^{2/3}}\\
& \leq c(\alpha,u_0)(1+M)^{cC(p)\beta^{2/3}}.
\end{aligned}
\end{equation}
Similarly, this estimate of $e^{cs}$ and \eqref{equation52} leads to 
\begin{equation}
\label{estimEs0}
\begin{aligned}
\sqrt{E(s,y_0,f_0)} 
&\le   \biggl(\alpha+\beta\big(c^{3/2}+ \ln^{3/2}(1+C\Vert u_0\Vert_{H_0^1(\Omega)})\big)\biggr)e^{c s_0}e^{C(p)\alpha^{2/3}} (1+M)^{cC(p)\beta^{2/3}}\|u_0\|_{H_0^1(\Omega)}\\
& \leq c(\alpha,\beta,u_0)(1+M)^{cC(p)\beta^{2/3}}.
\end{aligned}
\end{equation}
\par\noindent
{\bf First case :} $p\in(0,1]$. Since $s\geq 1$, the constant $c_1(s)$ defined in \eqref{c1} satisfies $c_1(s)\leq \frac{C^{1+p}}{1+p}[g']_p$.Therefore, 
by combining  \eqref{equation44}, \eqref{u0linfty} and \eqref{estimEs0}, we get 
\begin{equation}
\nonumber
\begin{aligned}
\|y_{n+1}\|_{L^\infty(Q_T)} \le & c(\alpha,u_0)(1+M)^{cC(p)\beta^{2/3}}\\
&+Cm \max\biggl( \frac{1}{p}c(\alpha,\beta,u_0)(1+M)^{cC(p)\beta^{2/3}},\\
& \hspace{4.5cm}\frac{(1+p)}{p} [g^\prime]_p^{1/p} C^{(1+p)/p} c^2(\alpha,\beta,u_0)(1+M)^{2cC(p)\beta^{2/3}}\biggr)\\
\leq  & C(p,\alpha,u_0,[g']_p,\beta) (1+M)^{2cC(p)\beta^{2/3}}.
 \end{aligned}
 \end{equation}
 Now, if $\beta$ is small enough so that $2cC(p)\beta^{2/3}<1$, the real $M_0$ defined as follows 
 \begin{equation}
 M_0:=\inf\big\{M>0 \mid C(p,\alpha,u_0,[g']_p,\beta) (1+M)^{2cC(p)\beta^{2/3}}\leq M\big\} \label{def_M0}
 \end{equation}
 exists and is independent of $n$. Moreover, for all $M=M_0$ and $s=\max(C(p)\psi(M_0)^{2/3},s_0)$ :
$$ 
\|y_{n+1}\|_{L^\infty(Q_T)}
 \le M.
 $$
 \par\noindent
 {\bf Second case :} $p=0$. In this case, $c_1(s)=  Cs^{-3/2} [g']_0=2Cs^{-3/2}\Vert g^\prime\Vert_\infty<1$ for $s$ large enough. By combining \eqref{equation44bis}, \eqref{u0linfty} and \eqref{estimEs0}, we get
 $$
 \begin{aligned}
\|y_{n+1}\|_{L^\infty(Q_T)}
 &\le c(\alpha,u_0)(1+M)^{cC(0)\beta^{2/3}}+ \frac{1}{1-c_1(s)}c(\alpha,\beta,u_0)(1+M)^{cC(0)\beta^{2/3}}. \\
  &\le  C(\alpha,u_0,[g']_0,\beta) (1+M)^{cC(0)\beta^{2/3}}
 \end{aligned}
$$
and we conclude as in the previous case. 
\end{proof}

We are now in position to prove by induction the following decay result for the sequence $(E(s,y_k,f_k))_{(k\in \mathbb{N})}$.

\begin{prop}\label{convergenceEk}
Assume that $g$ satisfies \ref{growth_condition} and \ref{constraint_g_holder} for some $p\in [0,1]$. Assume moreover that $2cC(p)\beta^{2/3}<1$.
Let $M_0$ be given by (\ref{def_M0}) and $s=\max(C(p)\psi(M_0)^{2/3},s_0)$. Let $(y_0,f_0)\in \mathcal{A}(s)$ be the solution of the extremal problem \eqref{extrema-problem} in the linear situation for which $g\equiv 0$. Then the sequence $(y_k,f_k)_{k\in \mathbb{N}}\in \mathcal{A}(s)$ defined by (\ref{algo_LS_Y}) satisfies
$$
\|y_k\|_{L^\infty(Q_T)}\le M_0, \quad \forall k\in \N.
$$
Moreover, the sequence $(E(s,y_k,f_k))_{k\in \N}\to 0$ tends to $0$ as $k\to \infty$. The convergence is at least linear, and is at least of order $1+p$ after a finite number of iterations. 
\end{prop}
\begin{proof} The uniform boundedness of the sequence $(y_k)_{k\in \mathbb{N}}$ follows by induction from Proposition \ref{Estimation-2-yk+1} and implies the decay to $0$ of $E(s,y_k,f_k)$. Remark that, from the construction of $M_0$, $\Vert y_0\Vert_{L^\infty(Q_T)}< M_0$.

\vskip 0.25cm
\par\noindent
{\bf First case :} $p\in(0,1]$.
From the definition of $p_k$ given in (\ref{pk}) 
we have $p_k(\widetilde{\lambda}_k):=\min_{\lambda\in[0,m]}p_k(\lambda)\le p_k(1)=c_1 (s) E(s,y_k,f_k)^{p/2} $
and thus
\begin{equation}
\label{C1Ebis}
c_2(s)\sqrt{E(s,y_{k+1},f_{k+1})}\le  \Big(c_2(s) \sqrt{E(s,y_k,f_k)}\Big)^{1+p}, \quad c_2(s):=c_1^{1/p}(s).
\end{equation}
Thus, if $c_2(s)\sqrt{E(s,y_0,f_0)}< \frac1{(p+1)^{1/p}}$, then $c_2(s)\sqrt{E(s,y_k,f_k)}\to 0$  as $k\to \infty$ with a rate $1+p$. On the other hand, if $c_2(s)\sqrt{E(s,y_0,f_0)}\ge \frac1{(p+1)^{1/p}}$, then  $I:=\{k\in \mathbb{N},\ c_2(s)\sqrt{E(s,y_k,f_k)}\ge \frac1{(p+1)^{1/p}}\}$ is a finite subset of $\mathbb{N}$. Indeed,  for all $k\in I$ and for all $\lambda\in[0,1]$ : $p_k'(s,\lambda)=-1+(p+1)\lambda^p c_1(s) E(s,y_k,f_k)^{p/2}$ and thus $p_k'(s,\lambda)=0$ if and only if $\lambda =\frac{1}{(p+1)^{1/p}c_2(s)\sqrt{E(s,y_k,f_k)}}$, which gives
$$
\begin{aligned}
p_k(s,\widetilde{\lambda_k})=\min_{\lambda\in[0,m]}p_k(\lambda)
&=\min_{\lambda\in[0,1]}p_k(\lambda)=p_k\Big(\frac{1}{(1+p)^{1/p}c_2(s)\sqrt{E(s,y_k,f_k)}}\Big)\\
&=1- \frac{p}{(1+p)^{\frac1p+1}}\frac{1}{c_2(s)\sqrt{E(s,y_k,f_k)}}
\end{aligned}$$
and thus
\begin{equation} \label{decayEunquart}
\begin{aligned}
c_2(s)\sqrt{E(s,y_{k+1},f_{k+1}) } 
&\le \Big(1-\frac{p}{(1+p)^{\frac1p+1}}\frac{1}{c_2(s)\sqrt{E(s,y_k,f_k)}}\Big)c_2(s)\sqrt{E(s,y_k,f_k)}\\
&=c_2(s)\sqrt{E(s,y_k,f_k) }-\frac{p}{(1+p)^{\frac1p+1}}.
\end{aligned}
\end{equation}
This inequality implies that the sequence $\big(c_2(s)\sqrt{E(s,y_k,f_k)}\big)_{k\in I}$ strictly decreases so that there exists $k_0\in I$ such that  
$c_2(s)\sqrt{E(s,y_{k_0+1},f_{k_0+1}}) <\frac1{(p+1)^{1/p}}$.
Thus the sequence  $\big(c_2(s) \sqrt{E(s,y_k,f_k)}\big)_{k\in \mathbb{N}}$ decreases to $0$ at least linearly and there exists 
$k_0\in\mathbb{N}$ such that for all $k> k_0$, $c_2(s)\sqrt{E(s,y_k,f_k)}<\frac1{(p+1)^{1/p}}$, that is  $I$ is a finite subset of $\mathbb{N}$. Arguing as in the first case, it follows that $c_2(s)\sqrt{E(s,y_k,f_k)}\to 0$ as $k\to \infty$. 

\vskip 0.25cm
\par\noindent
{\bf Second case :} $p=0$. Then for all $k\in\N$,  since $c_1(s)<1$, $p_k(s,\widetilde{\lambda_k})=c_1(s)$ (since $\widetilde{\lambda_k}=1$) and therefore
\begin{equation}
\sqrt{E(s,y_{k+1},f_{k+1})}\le   c_1(s) \sqrt{E(s,y_k,f_k)}\le c_1(s)^{k+1} \sqrt{E(s,y_0,f_0)} \label{estimp0}
\end{equation}
Thus $\sqrt{E(s,y_k,f_k)}\to 0$  as $k\to \infty$.
\end{proof}

We now prove the main result of this section.

 \begin{theorem}\label{ths1}
 Assume that $g$ satisfies \ref{growth_condition} and \ref{constraint_g_holder} for some $p\in [0,1]$. Assume moreover that $\beta$ is small enough so that 
 $$
 2cC(p)\beta^{2/3}<1
 $$
 with $c=\Vert \varphi(\cdot,0)\Vert_{L^\infty(\Omega)}$. Let $M_0$ be given by \eqref{def_M0} and $s=\max(C(p)\psi(M_0)^{2/3},s_0)$. Let $(y_0,f_0)\in \mathcal{A}(s)$ be the solution of the extremal problem \eqref{extrema-problem} in the linear situation for which $g\equiv 0$ and let $(y_k,f_k)_{k\in\mathbb{N}}$ be the sequence defined by (\ref{algo_LS_Y}). Then, $(y_k,f_k)_{k\in \mathbb{N}}\to (y,f)$ in $\mathcal{A}(s)$ where $f$ is a null control for $y$ solution of  \eqref{heat-NL}. 
The convergence is at least linear, and is at least of order $1+p$ after a finite number of iterations. 
\end{theorem}

\begin{proof}   For all $k\in \mathbb{N}$, let $F_k=- \sum_{n=0}^k \lambda_n F_n^1$ and $Y_k= \sum_{n=0}^k \lambda_n Y_n^1$. Let us prove that  $\big ((Y_k,F_k )\big)_{k\in\mathbb{N}}$ converges in $\mathcal{A}_0(s)$, i.e. that the series $\sum \lambda_n (F_n^1,Y_n^1)$ converges in $ \mathcal{A}_0(s)$. 
Using that $\Vert(Y_k^1, F_k^1)\Vert_{\mathcal{A}_0(s)}\leq C(M_0) \sqrt{E(s,y_k,f_k)}$ for all $k\in\N$ (see \eqref{estimateF1Y1}), we write, using \eqref{52} and \eqref{sommeE} :
$$
\sum_{n=0}^k\lambda_n\Vert(Y_n^1, F_n^1)\Vert_{\mathcal{A}_0(s)}\le m \sum_{n=0}^k\Vert(Y_n^1, F_n^1)\Vert_{\mathcal{A}_0(s)}
\le C(M_0) \sum_{n=0}^k\sqrt{E(s,y_n,f_n)}\leq \sqrt{E(s,y_0,f_0)} \frac{C(M_0)}{1-p_0(\widetilde{\lambda_0})}.
$$
We deduce that  the series $\sum_n \lambda_n (Y_n^1,F_n^1)$ is normally convergent and so convergent.
Consequently, there exists $(Y,F)\in\mathcal{A}_0(s)$ such that $(Y_k,F_k)_{k\in\N}$ converges to $(Y,F)$ in $\mathcal{A}_0(s)$.

Denoting $y=y_0+Y$ and $f=f_0+F$, we then have that $(y_k,f_k)_{k\in\N}=(y_0+Y_k,f_0+F_k)_{k\in\N}$ converges to $(y,f)$ in $\mathcal{A}(s)$.

It suffices now to verify that the limit $(y,f)$ satisfies $E(s,y,f)=0$. Using that $(Y^1_k, F^1_k)$ goes to zero in $\mathcal{A}_0(s)$ as $k\to \infty$, we pass to the limit in (\ref{heat-Y1k})
and get that $(y,f)\in \mathcal{A}(s)$ solves (\refeq{heat-NL}), that is $E(s,y,f)=0$. Moreover, we have 
\begin{equation}\label{estim_coercivity}
\Vert (y,f)-(y_k,f_k)\Vert_{\mathcal{A}_0(s)} \leq C(M_0)\sqrt{E(s,y_k,f_k)}, \quad \forall k>0
\end{equation}
which implies, using Proposition \ref{convergenceEk}, the announced order of convergence.  Precisely, 
\begin{equation}
\nonumber
\begin{aligned}
\Vert (y,f)-(y_k,f_k)\Vert_{\mathcal{A}_0(s)} & =\big\Vert \sum_{p=k+1}^{\infty} \lambda_p (Y^1_p,F^1_p)\big\Vert_{\mathcal{A}_0(s)}\leq m\sum_{p=k+1}^{\infty}  \Vert (Y^1_p,F^1_p) \Vert_{\mathcal{A}_0(s)}\\ 
& \leq m\, C(M_0)\sum_{p=k+1}^{\infty}  \sqrt{E(s,y_p,f_p)}\\
& \leq m\, C(M_0)\sum_{p=k+1}^{\infty}  p_{0}(\widetilde{\lambda}_0)^{p-k}\sqrt{E(s,y_k,f_k)}\\
& \leq m\, C(M_0)\frac{p_{0}(\widetilde{\lambda}_0)}{1-p_{0}(\widetilde{\lambda}_0)}\sqrt{E(s,y_{k},f_{k})}.
\end{aligned}
\end{equation}
\end{proof}

We emphasize, in view of the non uniqueness of the zeros of $E$, that an estimate (similar to \eqref{estim_coercivity}) of the form  $\Vert (y,f)-(\overline{y},\overline{f})\Vert_{\mathcal{A}_0(s)} \leq C(M_0) \sqrt{E(s,\overline{y},\overline{f})}$ does not hold for all $(\overline{y},\overline{f})\in\mathcal{A}(s)$. We also mention the fact that the sequence $(y_k,f_k)_{k\in \mathbb{N}}$ and its limits $(y,f)$ are uniquely determined from the initial guess $(y_0,f_0)$ and from our criterion of selection of the pair $(Y_k^1,F^1_k)$ for every $k$. In other words, the solution $(y,f)$ is unique up to the element $(y_0,f_0)$ and the functional $J$.

We also have the following convergence of the optimal sequence $(\lambda_k)_{k\in \mathbb{N}}$. 

\begin{lemma}\label{lambda-k-go-1}
Under hypotheses of Theorem \eqref{ths1} with $p\in (0,1]$, the  sequence $(\lambda_k)_{k\in \mathbb{N}}$ defined in (\ref{algo_LS_Y}) converges to $1$ as $k\to \infty$.
\end{lemma}
\begin{proof}  
If $p\in(0,1]$, in view of \eqref{E_expansionb}  we have, as long as  $E(y_k,f_k)>0$, since $\lambda_k\in[0,m]$ 
$$\begin{aligned}
(1-\lambda_k)^2
&=\frac{E(s,y_{k+1},f_{k+1})}{E(s,y_{k},f_k)}-(1-\lambda_k)\frac{\big( \rho_0(s)\big(\partial_ty_{k}+\Delta y_k+g(y_k)-f_k\,1_\omega\big),\rho_0(s)l(y_k,-\lambda_k Y_k^1) \big)_{L^2(Q_T)}}{E(s,y_{k},f_k)}\\
& \hspace{3cm}- \frac{\bigl \Vert\rho_0(s)l(y_k,-\lambda_k Y_k^1)\bigr\Vert^2_{L^2(Q_T)}}{2E(s,y_{k},f_k)}\\
&\le \frac{E(s,y_{k+1},f_{k+1})}{E(s,y_{k},f_k)}-(1-\lambda_k)\frac{\big( \rho_0(s)\big(\partial_ty_{k}+\Delta y_k+g(y_k)-f_k\,1_\omega\big),\rho_0(s)l(y_k,-\lambda_k Y_k^1) \big)_{L^2(Q_T)}}{E(s,y_{k},f_k)}\\
&\le \frac{E(s,y_{k+1},f_{k+1})}{E(s,y_{k},f_k)}+\sqrt{2}m\frac{\sqrt{E(s,y_k,f_k)}\|\rho_0(s)l(y_k,-\lambda_k Y_k^1) \|_{L^2(Q_T)}}{E(s,y_{k},f_k)}\\
&\le \frac{E(s,y_{k+1},f_{k+1})}{E(s,y_{k},f_k)}+\sqrt{2}m\frac{\|\rho_0(s)l(y_k,-\lambda_k Y_k^1) \|_{L^2(Q_T)}}{\sqrt{E(s,y_k,f_k)}}.
\end{aligned}
$$
But, from \eqref{38},  \eqref{estimateF1Y1} and Remark \ref{remark2}, we infer that 
$$
\begin{aligned}
\|\rho_0(s)l(y_k,-\lambda_k Y_k^1) \|_{L^2(Q_T)}
&\le [g^\prime]_p\frac{\lambda_k^{p+1}}{p+1}\Vert Y\Vert_{L^\infty(Q_T)}^p \Vert\rho(s) Y\Vert_{L^2(Q_T)}\\
&\le\lambda_k^{p+1}\frac{Cs^{-3/2}e^{-\frac{3p}2s}}{p+1}  [g']_p E(s,y_k,f_k)^{\frac{p+1}2}\\
&\le C(s) m^{p+1} [g']_p E(s,y_k,f_k)^{\frac{p+1}2}
\end{aligned}
$$
and thus 
$$
(1-\lambda_k)^2\le \frac{E(s,y_{k+1},f_{k+1})}{E(s,y_{k},f_k)}+m^{p+2}  C(s)[g']_p(E(s,y_k,f_k)) ^{p/2}.
$$
Consequently, since $E(s,y_{k},f_k)\to 0$ and $\frac{E(s,y_{k+1},f_{k+1})}{E(s,y_{k},f_k)}\to 0$, we deduce that $(1-\lambda_k)^2\to 0$ as $k\to\infty$.
\end{proof} 

If $p=0$ and if $s$ is large enough, then $c_1(s)<1$ and $\widetilde{\lambda_k}=1$ for every $k\in \mathbb{N}$ leading to the decay of $(E(s,y_k,f_k))_{k\in\mathbb{N}}$ to $0$ (see \eqref{estimp0}). Moreover, estimate (\ref{estimW1}) implies that the sequence $(\lambda_k)_{k\in \N}$
with $\lambda_k=1$ for every $k$ also leads to the decay $(E(s,y_k,f_k))_{k\in\mathbb{N}}$ with an order at least linear. Whether or not this constant sequence if the optimal one (as defined in \eqref{algo_LS_Y}).

\begin{remark}
In Theorem \ref{ths1}, the sequence $(y_k,f_k)_{k\in \mathbb{N}}$ is initialized with the solution of minimal norm corresponding to $g\equiv 0$. This  natural choice in practice leads to a precise estimate of $\sqrt{E(s,y_0,f_0)}$ with respect to the parameter $s$. Many other pairs are available  such as for instance the pair $(y_0,f_0)=(y,0)=(\phi \,y^\star,0)$ constructed in Lemma \ref{existence_triplet} since it leads to the following estimate in term of $s$:
$$
\begin{aligned}
\sqrt{E(s,\phi y^\star,0)}&=\frac{1}{\sqrt{2}}\Vert \rho_0\big(\partial_t (\phi y^\star)-\Delta (\phi y^\star)+g(\phi y^\star)\big)\Vert_2=\frac{1}{\sqrt{2}}\Vert \rho_0(s)\big(\phi_t y^\star+g(\phi y^\star)\big)\Vert_2\\
& \leq \frac{T}{2}\Vert \rho_0(s)\Vert_{L^\infty(Q_{T/2})}\Vert \partial_t\phi\Vert_{L^\infty(Q_{T/2})} \Vert y^\star\Vert_{L^\infty(Q_{T/2})}+\Vert \rho_0(s) g(\phi y^\star))\Vert_2\\
& \leq C(T)\bigg(\Vert \rho_0(s)\Vert_{L^\infty(Q_{T/2})}\Vert u_0\Vert_{H^1(\Omega)}+\psi (\|\phi y^\star\|_{L^\infty(Q_{T/2})})\|\rho(s)\phi y^\star\|_{L^\infty(Q_{T/2})}\biggr)\\
& \leq C(T)\Vert \rho(s)\Vert_{L^\infty(Q_{T/2})} \Vert u_0\Vert_{H^1(\Omega)}\biggl(1+\psi (\|u_0\|_{L^\infty(\Omega)})\biggr)\\
& \leq C(T) \Vert u_0\Vert_{H^1(\Omega)}\biggl(1+\alpha+\beta\ln^{3/2}(1+\Vert u_0\Vert_{H^1(\Omega)})\biggr)e^{s\Vert\varphi \Vert_{L^\infty(Q_{T/2})}}.
\end{aligned}
$$
\end{remark}

\begin{remark}
As stated in Theorem \ref{ths1}, the convergence is at least of order $1+p$ after a number $k_0$ of iterations. Using \eqref{decayEunquart}, $k_0$
is given by 
\begin{equation}\label{defk0}
k_0 = \left\lfloor  \frac{1+p}{p} \left((1+p)^{1/p} c_2(s) \sqrt{E(s,y_0,f_0)} - 1 \right) \right\rfloor + 1 ,
\end{equation}
(where $\lfloor \cdot\rfloor$ is the integer part) if $(1+p)^{1/p} c_2(s) \sqrt{E(s,y_0,f_0)} - 1 >0$, and $k_0=1$ otherwise.

\end{remark}

\section{Comments} \label{sec:remarks}

Several comments are in order.

\paragraph{Asymptotic condition.}
The asymptotic condition \ref{growth_condition} on $g^\prime$ is slightly stronger than the asymptotic condition \ref{asymptotic_g} made in \cite{EFC-EZ}: this is due to our linearization of \eqref{heat-NL} which involves $r\to g^{\prime}(r)$ while the linearization \eqref{NL_z} in \cite{EFC-EZ} involves $r\to g(r)/r$. There exist cases covered by Theorem \ref{nullcontrollheatplus} in which exact controllability for \eqref{heat-NL} is true but that are not covered by Theorem \ref{ths1}. Note however that the example $g(r)=a+b r+ c r \ln^{3/2}(1+\vert r\vert)$, for any $a,b\in \mathbb{R}$ and for any $c>0$  small enough (which is somehow the limit case in Theorem \ref{nullcontrollheatplus}) satisfies \ref{growth_condition} as well as 
\ref{constraint_g_holder} for any $p\in [0,1]$. 

While Theorem \ref{nullcontrollheatplus} was established in \cite{EFC-EZ} by a nonconstructive fixed point argument, we obtain here, in turn, a new proof of the exact controllability of semilinear multi-dimensional wave equations, which is moreover constructive, with an algorithm that converges unconditionally, at least with order $1+p$.

\paragraph{Minimization functional.} The estimate  \eqref{estimateF1Y1} is a key point in the convergence analysis and is independent of the choice of the functional $J$ defined by $J(y,f)=\frac{1}{2}\Vert \rho_0(s) f\Vert^2_{L^2(q_T)}+\frac{1}{2}\Vert \rho(s) y \Vert^2_{L^2(Q_T)}$  (see Proposition \ref{controllability_result}) in order to select a pair $(Y^1,F^1)$ in $\mathcal{A}_0(s)$. Thus, we may consider other weighted functionals, for instance $J(y,f)=\frac{1}{2}\Vert \rho_0(s) f\Vert^2_{L^2(q_T)}$ as discussed in \cite{DeSouza_Munch_2016}.

\paragraph{Link with Newton method.}
If we introduce $F:\mathcal{A}(s)\to L^2(Q_T)$ by $F(y,f):=\rho_0^{-1}(s)(\partial_t y-\partial_{xx} y + g(y)-f\,1_\omega)$, we get that $E(s,y,f)=\frac{1}{2}\Vert F(y,f)\Vert_{L^2(Q_T)}^2$ and check that, for $\lambda_k=1$, the algorithm \eqref{algo_LS_Y} coincides with the Newton algorithm associated with the mapping $F$. This explains the super-linear convergence in Theorem \ref{ths1}. The optimization of the parameter $\lambda_k$ is crucial here as it allows to get a global convergence result. Its leads to so-called damped Newton method (for $F$) (we refer to \cite[Chapter 8]{deuflhard}). As far as we know, the analysis of damped type Newton methods  for partial differential equations has deserved very few attention in the literature. We mention \cite{lemoinemunch_time, saramito} in the context of fluid mechanics.

\paragraph{A variant.}
To simplify, let us take $\lambda_k=1$, as in the standard Newton method. Then, for each $k\in\N$, the optimal pair $(Y_k^1,F_k^1)\in \mathcal{A}_0$ is such that the element $(y_{k+1},f_{k+1})$ minimizes over $\A(s)$ the functional $(z,v)\to J(z-y_k,v-f_k)$. 
Alternatively, we may select the pair $(Y_k^1,F_k^1)$ so that the element $(y_{k+1},f_{k+1})$ minimizes the functional $(z,v)\to J(z,v)$. This leads to the sequence $(y_k,f_k)_{k\in\N}$ defined by 
 \begin{equation}
\label{eq:wave_lambdakequal1}
\left\{
\begin{aligned}
& \partial_{t}y_{k+1} - \partial_{xx} y_{k+1} +  g^{\prime}(y_k) y_{k+1} = f_{k+1} 1_{\omega}+g^\prime(y_k)y_k-g(y_k) & \textrm{in}\  Q_T,\\
& y_k=0  \,\,\,\textrm{on}\,\,\,  \Sigma_T, \quad y_{k+1}(\cdot,0)=u_0 \,\,\,\textrm{in}\,\,\,  \Omega.
\end{aligned}
\right.
\end{equation}
In this case, for every $k\in\N$, $(y_k,f_k)$ is a controlled pair for a linearized heat equation, while, in the case of the algorithm \eqref{algo_LS_Y}, $(y_k,f_k)$ is a sum of controlled pairs $(Y^1_j,F^1_j)$ for $0\leq j\leq k$.
This analysis of this variant used in \cite{EFC-AM} is apparently less straightforward.
%

%

\paragraph{Local controllability when removing the growth condition \ref{growth_condition}.}
As in \cite{munch_lemoine_waveND, munch_trelat} devoted to the wave equations, we may expect to remove the growth condition \ref{growth_condition} on $g^\prime$ if the initial value $E(s,y_0,f_0)$ is small enough. For $s$ fixed, in view of Lemma \ref{y0f0g0}, this is notably true if $g(0)=0$ and if the norm $\Vert u_0\Vert_{H_0^1(\Omega)}$ of the initial data to be controlled is small enough. This would allow to recover the local controllability of the heat equation (usually obtained by an inverse mapping theorem, see \cite[chapter 1]{fursikov-imanuvilov}) and would be in agreement with the usual convergence of the standard Newton method. In the parabolic case considered here, the proof is however open, since in order to prove the convergence of $(E(s,y_k,y_k))_{k\in\mathbb{N}}$ to zero, for some $s$ large enough independent of $k$, we need to prove that the sequence $(\Vert y_k\Vert_{L^\infty(Q_T)})_{k\in \mathbb{N}}$ is bounded. This is in contrast with the wave equation where the parameter $s$ does not appear.

%

\paragraph{Weakening of the condition \ref{constraint_g_holder}.}

Given any $p\in [0,1]$, we introduce for any $g\in C^1(\R)$ the following hypothesis : 
\begin{enumerate}[label=$\bf (\overline{H}^{\prime}_p)$,leftmargin=1.5cm]
\item\label{constraint_g_holderbis}\ There exist $\overline{\alpha},\overline{\beta},\gamma\in \R^+$ such that 
$\vert g^\prime(a)-g^\prime(b)\vert \leq \vert a-b\vert^p \big(\overline{\alpha}+\overline{\beta}(\vert a \vert^\gamma +\vert b\vert^\gamma)\big), \quad \forall a,b\in \R$
\end{enumerate}
which coincides with \ref{constraint_g_holder} if $\gamma=0$ for $\overline{\alpha}+2\overline{\beta}=[g^\prime]_p$. If $\gamma\in (0,1)$ is small enough and related to the constant $\beta$ appearing in the growth condition \ref{growth_condition}, Theorem \ref{ths1} still holds if \ref{constraint_g_holder} is replaced by the weaker hypothesis \ref{constraint_g_holderbis}. 
%

\paragraph{Influence of the parameter $s$ and a simpler linearization.}
Taking $s$ large enough in the case $p=0$ (corresponding to $g^\prime\in L^\infty(\R)$) allows to ensure that the coefficient $c_1(s)$ (see \eqref{c1}) is strictly less than one, and then to prove the strong convergence of the sequence $(y_k,f_k)_{k\in \mathbb{N}}$. This highlights  the influence of the parameter $s$ appearing in the Carleman weights $\rho$, $\rho_0$ and $\rho_1$. Actually, in this case, a similar convergence can be obtained by considering a simpler linearization of the system (\ref{heat-NL}). For any $s\geq s_0$ and $z\in L^2(\rho(s), Q_T)$, we define the controlled pair $(y,f)\in \mathcal{A}(s)$ solution of  
\begin{equation}
\nonumber
\left\{
\begin{aligned}
& \partial_ty - \partial_{xx} y =  f 1_{\omega}-g(z) \quad  \textrm{in}\quad Q_T,\\
& y=0 \,\,\, \textrm{on}\,\,\, \Sigma_T, \quad y(\cdot,0)=u_0 \,\,\, \textrm{in}\,\,\, \Omega,
\end{aligned}
\right.
\end{equation}
and which minimizes the weighted cost $J$. If $g(0)=0$ and $g$ is globally Lipschitz, then $\rho_0(s) g(z)\in L^2(Q_T)$ and Theorem \ref{controllability_result} implies  $\Vert \rho(s)\, y\Vert_{L^2(Q_T)}\leq  Cs^{-3/2} \big(\|\rho_0(s)g(z)\|_{L^2(Q_T)} +e^{cs }\|u_0\|_2 \big)$.
This allows to define the operator $K:L^2(\rho(s),Q_T)\to L^2(\rho(s),Q_T)$ by $y:=K(z)$. From Lemma \ref{solution-controle}, for any $z_i\in L^2(\rho(s),Q_T)$, $i=1,2$, $y_i:=K(z_i)$ is given 
by $y_i=\rho^{-2}(s)L^\star_0 p_i$ where $p_i\in P_0$ solves 
\begin{equation}
\nonumber
(p_i,q)_P=\int_\Omega u_0q(0)-\int_{Q_T}  g(z_i) q ,\quad \forall q\in P_0.
\end{equation}
Taking $q:=p_1-p_2$, we then get 
$$
\begin{aligned}
\Vert p_1-p_2\Vert_P^2 & \leq  \int_{Q_T} \vert g(z_1)-g(z_2)\vert \vert p_1-p_2\vert \\
& \leq \Vert \rho_0(s)(g(z_1)-g(z_2))\Vert_2 \Vert \rho^{-1}_0(s)(p_1-p_2)\Vert_2\\
& \leq \Vert g^\prime\Vert_{L^\infty(\R)}\Vert \rho_0(s)(z_1-z_2)\Vert_2 \Vert \rho^{-1}_0(s)(p_1-p_2)\Vert_2.
\end{aligned}
$$
Using that $\rho_0\leq \rho$ and Lemma \ref{carleman}, we obtain 
$$
\big\Vert \rho(s)\big(K(z_1)-K(z_2)\big)\big\Vert_2\leq \Vert g^\prime\Vert_\infty C \lambda_0^4 s^{-3/2} \Vert \rho(s)(z_1-z_2)\Vert_2, \quad \forall z_1,z_2\in L^2(\rho(s),Q_T)
$$
and conclude that, if $s> \max(s_0, (C\lambda_0^4 \Vert g^\prime\Vert_\infty)^{2/3})$, then the operator $K$ is contracting. This allows to infer the convergence of the sequence $(y_{k})_{k\in \mathbb{N}}$ defined by $y_{k+1}=K(y_k)$, $k\geq 0$ for any $y_0\in L^2(\rho(s),Q_T)$ to a controlled solution of \eqref{heat-NL}. On order to replace the assumption $g^\prime\in L^\infty(\R)$ by \ref{asymptotic_g}, one needs to show some compactness properties for $K$, which is an open question. 

\paragraph{The linearization \eqref{NL_z} associated with the weighted cost $J$.}
Similarly, one can wonder if a parameter $s$ large enough may leads to a contracting property for the operator $\Lambda$ introduced in \cite{EFC-EZ} and leading to the linearization \eqref{NL_z}. For any $\beta>0$, we introduce the hypothesis
\begin{enumerate}[label=$\bf (H_1^\prime)$,leftmargin=1.5cm]
\item\label{asymptotic_g_bis}\ $\limsup_{\vert r\vert \to \infty} \frac{\vert g(r)\vert }{\vert r\vert \ln^{3/2}\vert r\vert}\leq \beta$
\end{enumerate}
similar to \ref{asymptotic_g}. Then, as \cite{EFC-EZ}, the linearization \eqref{NL_z} also leads to a compactness property when associated with the weighted cost $J$. 

\begin{prop}\label{TzMM}
Assume that $g$ satisfies \ref{asymptotic_g_bis} with $c\beta^{2/3}<1$ with $c=\Vert \varphi(0,\cdot)\Vert_\infty$. Let $z\in L^\infty(Q_T)$ and $s\geq \max(\Vert \widetilde{g}(z)\Vert^{2/3}_\infty,s_0)$.
Let $(y,f)\in \A(s)$ the minimizer of the functionnel $J$ and solution of 
\begin{equation}
\nonumber
\left\{
\begin{aligned}
& \partial_t y - \partial_{xx} y +  y\,\widetilde{g}(z)= f 1_{\omega} \quad  \textrm{in}\quad Q_T,\\
& y=0 \,\,\, \textrm{on}\,\,\, \Sigma_T, \quad y(\cdot,0)=u_0 \,\,\, \textrm{in}\,\,\, \Omega.
\end{aligned}
\right. 
\end{equation}
There exists $M>0$ such that if $\Vert z\Vert_{\infty}\leq M$ and $s:=\max(\psi(M)^{2/3},s_0)$ with $\psi(r)=\alpha+\beta\ln^{3/2}(1+\vert r\vert)$,
then  $\Vert y\Vert_{L^\infty(Q_T)}\leq M$. We note $\Lambda_s:L^\infty(Q_T)\to L^\infty(Q_T)$ such that $y=\Lambda_s(z)$.
\end{prop}
\begin{proof}
 \ref{asymptotic_g_bis} implies that  $\vert \widetilde{g}(r)\vert\leq \psi(r):=\alpha+\beta \ln^{3/2}(1+\vert r\vert)$ for all $r\in \R$.
Consequently, $\Vert \widetilde{g}(z)\Vert_{L^\infty(Q_T)}\leq \psi(\Vert z\Vert_{L^\infty(Q_T)})=\alpha+\beta \ln^{3/2}(1+\Vert z\Vert_{L^\infty(Q_T)})$ leading to 
$c\Vert \widetilde{g}(z)\Vert_\infty^{2/3}\leq c\alpha^{2/3}+c\beta^{2/3} \ln(1+\Vert z\Vert_{L^\infty(Q_T)})$ and then to 
$e^{c s}\leq e^{c(s_0+\alpha^{2/3})} (1+\Vert z\Vert_{L^\infty(Q_T)})^{c\beta^{2/3}}$.
Estimate \eqref{estimation1} then implies 
$
\Vert \rho(s)\, y\Vert_{L^2(Q_T)}+ \Vert \rho_0(s)\,f\Vert_{L^2(q_T)} \leq  e^{c(s_0+\alpha^{2/3})} (1+\Vert z\Vert_\infty)^{c\beta^{2/3}}\|u_0\|_2
$
and in particular, since $\rho(s),\rho_0(s)\geq 1$ that 
\begin{equation}
\label{estimyz}
\Vert y\Vert_{L^2(Q_T)}+ \Vert f\Vert_{L^2(q_T)} \leq  e^{c(s_0+\alpha^{2/3})} (1+\Vert z\Vert_{L^\infty(Q_T)})^{c\beta^{2/3}}\|u_0\|_2.
\end{equation}
Moreover, if $u_0\in H_0^1(\Omega)$, standard estimate for the heat equation reads as 
$$
\Vert y\Vert_{L^\infty(Q_T)} \leq C \biggl(\Vert \widetilde{g}(z)\Vert_{\infty} \Vert y\Vert_{L^2(Q_T)}+\Vert f\Vert_{L^2(q_T)}+ \Vert u_0\Vert_2 \biggr)
$$
which combined with \eqref{estimyz} leads to 
$$
\Vert y\Vert_{L^\infty(Q_T)} \leq C \biggl(\big(1+ \alpha+\beta \ln^{3/2}(1+\Vert z\Vert_{L^\infty(Q_T)})\big) e^{c(s_0+\alpha^{2/3})} (1+\Vert z\Vert_{L^\infty(Q_T)})^{c\beta^{2/3}}\|u_0\|_2+ \Vert u_0\Vert_2 \biggr).
$$
It follows that if $c\beta^{2/3}<1$, then there exists an $M>0$  depending on $\Vert u_0\Vert_2, \alpha,s_0, \Omega,T$ such that 
$\Vert z\Vert_{L^\infty(Q_T)}\leq M$ implies $\Vert y\Vert_{L^\infty(Q_T)}\leq M$  since 
$$
\frac{C \biggl((1+ \alpha+\beta \ln^{3/2}(1+M)) e^{c(s_0+\alpha^{2/3})} (1+M)^{c\beta^{2/3}}\|u_0\|_2+ \Vert u_0\Vert_2 \biggr)}{M}\longrightarrow 0^+, \quad\textrm{as}\quad M\to\infty.
$$
\end{proof}

Let now $z_i\in L^2(\rho(s),Q_T)$, $i=1,2$, $y_i:=\Lambda_{\rho}(z_i)$. Then, using the estimates of Theorem \ref{controllability_result} and the characterization of Lemma \ref{solution-p}, we can proved, for all $s\geq \max(\psi(M)^{2/3},s_0)$ that 
\begin{equation}
\label{estim_Lambda_phi}
\Vert \rho(s) \big(\Lambda_s(z_1)-\Lambda_s(z_2)\big)\Vert_{L^2(Q_T)} \leq C\Vert \widetilde{g}^{\prime}\Vert_{L^\infty(0,M)}\Vert u_0\Vert_{L^2(\Omega)} \Vert \zeta\rho^{-1}(s)\Vert_{L^\infty(Q_T)} s^{-3/2} e^{cs}\Vert \rho(z_1-z_2)\Vert_{L^2(Q_T)}
\end{equation}
for some $C=C(\Omega,T)$  and $\zeta=(T-t)^{-1/2}$. The existence of a parameter $s$ large enough for which the operator $\Lambda_s$ enjoys a contracting property remains however an open issue.

\section{Conclusions}\label{conclusion}

Exact controllability of \eqref{heat-NL} has been established in \cite{EFC-EZ}, under a growth condition on $g$, by means of a Kakutani fixed point argument that is not constructive. Under the slightly stronger growth condition 
and under the additional assumption that $g^\prime$ is uniformly H\"older continuous with exponent $p\in[0,1]$, we have designed an explicit algorithm and proved its convergence to a controlled solution of \eqref{heat-NL}. Moreover, the convergence is super-linear of order greater than or equal to $1+p$ after a finite number of iterations. In turn, our approach gives a new and constructive proof of the exact controllability of \eqref{heat-NL}, which is, at least in the one-dimensional setting, simpler that in \cite{EFC-EZ} where refined $L^1$ Carleman estimates are employed. In fact, in the one-dimensional setting, we can achieve the power $3/2$ appearing in the hypothesis \ref{growth_condition} since the controlled sequence $(y_k)_{k\in \mathbb{N}}$ belongs to the space $L^{\infty}(Q_T)$. This is in general no longer true in the multidimensional setting in view of the $L^2(Q_T)$ right hand side term in (\ref{heat-Y1k}), even with $L^{\infty}(q_T)$ controls. Therefore, whether or not we can achieve the power 3/2 in the multidimensional case (as in \cite{EFC-EZ} with $L^{\infty}(q_T)$ controls but a different linearization) through our least-squares approach is an open and interesting question. 

We also emphasize that the method is general and may be applied to any other equations or systems for which a precise observability estimate for the linearized problem is available. Such estimates are usually obtained by the way of Carleman estimates as initially done in the monography of  Imanuvilov-Fursikov \cite{fursikov-imanuvilov}, extended later to a very large number of systems and situations. For instance, the method can be extended to the case of boundary controls. This remains however to be done. Moreover, the introduction of the Carleman type weights, which blow up at the final time and which depends on several parameters (itself related to the controlled solution), makes the analysis quite intricate. From this point of view, the case of hyperbolic equations (considered in \cite{munch_trelat, munch_lemoine_waveND}) is simpler. Whether or not an appropriate choice of these parameters may lead directly to some contracting properties for some fixed point operator is also an open and interesting issue. Eventually, it would be also interesting to address other types of linearity involving notably the gradient of the solution  (see \cite{burgos}): we mention notably the Burgers equation and the Navier-Stokes system, formally solved numerically from a controllability viewpoint in \cite[Part 1]{glowinski-book} and in \cite{diegoNS} respectively.

{\small
 \bibliographystyle{plain}
 \bibliography{heatLS.bib}
 }
\end{document}